\documentclass[12pt]{amsart}
\usepackage{microtype}
\usepackage{mathtools,amssymb}
\usepackage{newtxtext,newtxmath} 
\usepackage{import}
\usepackage[utf8]{inputenc}
\usepackage{tikz-cd}
\usepackage{amssymb}
\usepackage{amsmath}
\usepackage{dsfont}
\usepackage{enumitem}
\usepackage{bm}
\usepackage{bbm}
\usepackage{braket}

\usepackage{xcolor}
\usepackage{tikz}

\newcommand{\M}{\mathfrak{M}}

\newcommand{\lip}{L}
\newcommand{\A}{\mathfrak{A}}
\newcommand{\BB}{\mathfrak{B}}
\newcommand{\dom}{\mathrm{dom}}
\newcommand{\B}{\mathbb{B}}
\newcommand{\K}{\mathbb{K}}
\newcommand{\HH}{\mathcal{H}}
\newcommand{\KK}{\mathcal{K}}
\newcommand{\supp}{\mathrm{supp}}
\newcommand{\multiplier}{\mathcal{M}}
\newcommand{\propagation}{{\rm prop}}
\newcommand{\V}{\mathcal{V}}
\newcommand{\propagationS}{\mathrm{prop}}

\newcommand{\spec}{\mathrm{spec}}
\newcommand{\comm}{\mathrm{comm}}
\newcommand{\asdim}{\mathrm{asdim}}

\newcommand{\diam}{\mathrm{diam}}
\newcommand{\unit}{\mathbbm{1}}
\newcommand{\sa}{\mathfrak{sa}}
\newcommand{\uu}{\mathfrak{u}}
\newcommand{\state}{\mathcal{S}}
\newcommand{\R}{\mathbb{R}}
\newcommand{\C}{\mathbb{C}}
\newcommand{\N}{\mathbb{N}}

\newtheorem{theorem}{Theorem}[section]
\newtheorem{lemma}[theorem]{Lemma}
\newtheorem{prop}{Proposition}[section]
\newtheorem{crl}{Corollary}[section]

\theoremstyle{definition}
\newtheorem{definition}[theorem]{Definition}
\newtheorem{example}[theorem]{Example}
\newtheorem{notation}[theorem]{Notation}

\newtheorem*{claim*}{Claim}
\newtheorem{ack}{Acknowledgments} 
\theoremstyle{remark}
\newtheorem{remark}[theorem]{Remark}
\numberwithin{equation}{section}

\usepackage{newtxtext,newtxmath}
\usepackage[margin=1in]{geometry}
\usepackage{amsaddr}
\usepackage{slashed}

\DeclareUnicodeCharacter{2217}{$\ast$}

\usepackage{cite}

\begin{document}
\title[]{Noncommutative coarse metric geometry}
\author{Ayoub Hafid}
\address{Graduate School of Mathematical Sciences, The University of Tokyo, Komaba, Meguro-ku, Tokyo
  153-8914, Japan}
\curraddr{}
\email{hafid@ms.u-tokyo.ac.jp}
\thanks{}

\maketitle

\begin{abstract}
  Motivated by coarse geometry and the classical role of Roe algebras as large-scale invariants of proper metric spaces, we show that proper quantum metric spaces as introduced by Latrémolière are noncommutative coarse spaces. This further allows us to develop a bridge between Latrémolière's framework and the W*-metric approach to quantum metric spaces.
  Furthermore, we construct Roe algebras for locally compact quantum metric spaces and verify that they recover the classical Roe algebras in the commutative case. We furthermore apply this framework to some examples of locally compact quantum metric spaces and show that it leads to the natural conclusion. Finally we use this framework to introduce notions of higher index theory for locally compact quantum metric spaces.
\end{abstract}
\section{Introduction}
\subsection{Quantum metric spaces}
Noncommutative geometry centers around the idea of generalizing geometric and topological notions to the setting of noncommutative C*-algebras. One of these structures is that of metric spaces. There have been two main approaches to define a noncommutative notion of metric spaces:

\textbf{Lipschitz seminorms approach:} The first approach has its roots in the work of Alain Connes \cite{connesNoncommutativeGeometry1994} where he introduced spectral triples as a noncommutative counterpart to Spin manifolds, and noted that this structure can be used to construct a metric on the state space of the underlying C*-algebra.

Indeed, if \((\A,\HH,D)\) is a spectral triple with \(\A\) a C*-algebra represented on a Hilbert space \(\HH\) and \(D\) the ``Dirac operator", an unbounded operator on \(\HH\) with compact resolvent, then the norm of the commutator \(\|[D,a]\|\) has properties of the Lipschitz seminorm of \(a\), and corresponds to the classical case \((C(M),L^2(M,S),D)\) of a spin manifold \(M\). This leads to the definition of a distance on the state space \(\mathcal{S}(\A)\) of \(\A\) by:
\[d(\mu,\nu)=\sup\{|\mu(a)-\nu(a)|:\, a\in \A, \|[D,a]\|\leq 1\}.\]

This approach was further axiomatized and developed by Rieffel \cite{Rieffel1999}, who introduced the notion of compact quantum metric spaces considering an arbitrary Lipschitz seminorm \(\lip\) on a unital C\(^\ast\)-algebra \(\A\). More recently, Latrémolière \cite{latremoliereQuantumLocallyCompact} extended this framework to locally compact quantum metric spaces. Finding examples of compact and noncompact quantum metric spaces is an active area of research due to the analytical condition that the distance \(d(\mu,\nu)=\sup\{|\mu(a)-\nu(a)|:\, a\in \A, \|[D,a]\|\leq 1\}\) should metrize the weak* topology of the set of states (or a subset of it in the case of locally compact quantum metric spaces).

\textbf{The Quantum relations approach:}
A different approach to define noncommutative metric spaces was introduced by Kuperberg and Weaver \cite{kuperbergNeumannAlgebraApproach2012} using the notion of quantum relations. In this framework, a quantum relation on a von Neumann algebra \(M\) is defined as a weak* closed operator bimodule over its commutant \(M'\). Kuperberg and Weaver then define a W*-quantum metric on \(M\) as a family of weak* closed operator systems \((\mathcal{V}_r)_{r\geq 0}\) satisfying properties analogous to those of operators of propagation less than or equal to \(r\) on a metric space. This approach has as natural examples quantum graphs (e.g. quantum expanders), and constructions related to quantum error correction in quantum computing. These examples are quite different from the "geometrical" ones arising in the Lipschitz seminorm approach, and a bridge between the two approaches becomes desirable.

In \cite{kuperbergNeumannAlgebraApproach2012}, starting from this definition spaces corresponding to \(\mathrm{Lip(X)}\) the space of Lipschitz functions have been constructed, with technical subtleties. Furthermore, a W*-quantum metric structure was associated to spectral triples. But the link between this approach and that of Rieffel/Latrémolière has not been fully explored.

\subsection{Coarse geometry}
The definition of \cite{kuperbergNeumannAlgebraApproach2012} is reminiscent of that of Roe algebras in the classical setting. These algebras are coarse invariants (they only care about the large scale structure of the underlying space) and have been widely studied and used in relation with index theory, the Novikov conjecture and coarse properties of the metric spaces \cite{roeIndexTheoryCoarse1996}\cite{willettHigherIndexTheory2020}.

In the classical case, we thus have:
\begin{itemize}
  \item Any proper metric space is a coarse space.
  \item For any coarse space, and a representation of the algebra of functions on a Hilbert space, one can associate a Roe algebra.
\end{itemize}
\[
  \begin{array}{cccc}
    \text{Proper Metric Spaces}
     & \xrightarrow{\text{forgetful functor} } &
    \text{Coarse Spaces}
     &
    \xrightarrow[\text{(not a functor)}]{\text{Roe algebra}}
    \textstyle C^{*}\text{-Algebras}
  \end{array}
\]
Based on the discussion above, we wish to replicate this construction to the noncommutative setting.
A notion of proper quantum metric spaces has been introduced in the framework of locally compact quantum metric spaces of Latrémolière \cite{latremoliereTopographicGromovHausdorffQuantum2014}, furthermore a notion of noncommutative coarse geometry has been developed in \cite{banerjeeNoncommutativeCoarse}. But the link between the two notions has not been yet established.

\subsection{Seminorms on Roe algebras: Relative commutants}

Again, in the case of \(X\) being a metric space, we have the following folklore inequality (see for instance \cite{willettHigherIndexTheory2020})
\[\|[T,f]\|\leq 8\propagationS(T)\|T\|\lip(f),\]
where \(\lip\) is the Lipschitz seminorm of \(f\in C_{b}(X)\) and \(T\) is a finite propagation operator. In some sense this inequality indicates a duality between \(\propagationS(T)\|T\|\) and \(\lip(f)\). More precisely, if we define:
\begin{align*}
  \lip^{\ast}(T)=\sup\{\|[T,f]\|:\, f\in C_0(X), \lip(f)\leq 1\},
\end{align*}
then \(\lip^{\ast}\) is a seminorm that satisfies:
\[
  \|[T,f]\|\leq \lip(f)\lip^{\ast}(T)
\]
and satisfies the Leibniz inequality \(\lip^{\ast}(ST)\leq\|S\|\lip^{\ast}(T)+\lip^{\ast}(S)\|T\|.\)

While the Roe algebra is contained within the C\(^{\ast}\) algebra generated by finite \(L^{\ast}\) seminorm locally compact operators, the inverse inclusion does not hold in general; in fact, the equality is closely related to a conjecture of Roe that was solved in \cite{spakulaRelativeCommutantPictures2019}. To what extent is this seminorm \(\lip^{\ast}\) related to the propagation?

\subsection{Summary of results}
In this work we give the following answers to the questions/program above:
\begin{itemize}
  \item \textbf{Theorem \ref{C*metricToW*metric}}: With the additional datum given by a representation on a Hilbert space, locally compact quantum metric spaces in the sense of Latrémolière \cite{latremoliereQuantumLocallyCompact} give rise to W*-metric spaces in the sense of Kuperberg and Weaver \cite{kuperbergNeumannAlgebraApproach2012}.
  \item \textbf{Theorem \ref{spechigsoncompact} and \ref{commHigsonCompact}}: With the additional datum given by a representation on a Hilbert space, W*-metric spaces, and thus locally compact quantum metric spaces are noncommutative coarse spaces in the sense of \cite{banerjeeNoncommutativeCoarse}.
  \item \textbf{Theorems \ref{spectralRoeAlgebraClassical} and \ref{HigsonCompactificationClassical}}: In the case of classical metric spaces, the spectral Roe algebra and the Higson compactification coincide with the classical notions.

  \item \textbf{Section \ref{examplesection}}: New examples of locally compact quantum metric spaces are provided: coarse disjoint unions, and almost-commutative locally compact quantum metric spaces, furthermore almost-commutative locally compact quantum metric spaces are shown to be coarsely equivalent to the underlying classical metric space: theorem \ref{almostcommutativecoarseequivalence}.

  \item \textbf{Theorem \ref{cutcommutationroe}}: A Higher index can be defined for locally compact quantum metric spaces, and under further assumptions on the asymptotic behavior of commutators of elements in these spaces, this higher index can be defined for all K-homology classes. Remarkably, this crucially uses the fundamental characterization of locally compact quantum metric spaces (theorem \ref{locallycompactqmscharacterization}).

  \item \textbf{Proposition \ref{higherindexspectraltriple}}: In the case of nonunital spectral triples, one can define the Higher index of the Dirac operator without further assumptions.

  \item \textbf{Theorem \ref{assouadnagatarelativecommutantroetheorem}} and \textbf{section \ref{sectionlocalizedrelativecommutantroetheory}}: As a side result relevant to classical setting, we show that relative commutant Lipschitz seminorm and finite propagation are essentially the same in the case of classical metric spaces with \emph{finite dimension at all scales} (Finite Assouad-Nagata dimension).

\end{itemize}

\tableofcontents

\begin{ack}
  The author would like to express his gratitude to his supervisor, Professor KAWAHIGASHI Yasuyuki, for his constant support and encouragement. He is also grateful to Professor KUBOTA Yosuke for valuable remarks and comments that significantly improved this work. Finally, he would like to thank Professor LATRÉMOLIÈRE Frédéric for insightful discussions during a short visit to the University of Denver.
  This work was partly supported by the Forefront Physics and Mathematics Program to Drive Transformation.

\end{ack}

\section{Preliminaries}
Here and there after we adopt the following notation conventions for the objects considered:
\begin{itemize}[label=--]
  \item \(\HH\), \(\KK\) denote Hilbert spaces,
  \item \(\A,\BB,\M \dots \)  denote C\(^*\)-algebras (underlying quantum metric spaces),
  \item \(X\) for locally compact spaces,
  \item \(M\) for von Neumann algebras.
\end{itemize}
Along with the following notations of standard constructions on these objects:
\begin{itemize}[label=--]
  \item \(\B(\HH)\) denotes the set of bounded operators on \(\HH\),
  \item \(\K(\HH)\) denotes the set compact operators on \(\HH\),
  \item \(\sigma(\M)\) denotes the Gelfand spectrum of an abelian C\(^*\)-algebra \(\M\),
  \item \(\sa(\A)\) denotes the self-adjoint part of a C\(^*\)-algebra \(\A\),
  \item \(\uu\A\) denotes the unitization of a C\(^*\)-algebra \(\A\),
  \item \(\multiplier(\A)\) denotes the multiplier algebra of a C\(^*\)-algebra \(\A\).
\end{itemize}

\subsection{Locally Compact Quantum Metric Spaces}
The notion of compact quantum metric spaces was defined, based on the formula of the distance between probability measures on compact metric spaces due to Kantorovich, and considered by Connes and  by Rieffel in \cite{Rieffel1999}. This notion was later extended to the locally compact case by Latrémolière in \cite{latremoliereQuantumLocallyCompact}. We present here this definition of locally compact quantum metric spaces. As we will be dealing mainly with proper quantum metric spaces, all the algebras \(\A\) and \(\M\) involved in this section and thereafter will be assumed to be separable.
Quantum metric spaces (at least in their C*-version) are defined by equipping a C\(^\ast*\)-algebra with a Lipschitz seminorms. Indeed in the classical case, the Lipschitz seminorms recover the distance.
\begin{definition}[Lipschitz pair]
  A \emph{Lipschitz pair} \((\A, \lip)\) is given by a C\(^*\)-algebra \(\A\) and a seminorm \(\lip\) defined on a dense subspace \(\dom(\lip)\) of the self-adjoint part of the unitization \(\sa(\uu\A)\) of \(\A\), such that:
  \begin{equation*}
    \{a\in\sa(\uu\A):\lip(a)=0\}=\mathbb{R}\unit_{\uu\A}
  \end{equation*}
\end{definition}
The seminorm \(\lip\) generalizes the notion of the Lipschitz seminorm for functions on a metric space. Given a Lipschitz pair \((\A, \lip)\), one can define a metric on the state space \(\mathcal{S}(\A)\) of \(\A\) by:
\begin{equation*}
  \mathrm{mk}_{\lip}(\mu, \nu) = \sup \{ |\mu(a) - \nu(a)| \mid a\in \mathfrak{sa}(\uu\A),  \lip(a) \leq 1 \}
\end{equation*}
Looking at the classical case, while for compact metric spaces, this metric metrizes the weak* topology on the state space, for locally compact metric spaces this is no longer true in general. This is because, in some sense, the fact that measures can escape to infinity prevents sequential compactness. Instead, Dobrushin \cite{dobrushinPrescribingSystemRandom1970a} shows that if one restricts to sets of measures satisfying a tightness condition, then the Monge-Kantorovich metric metrizes the weak* topology on these sets:
\begin{theorem}[Dobrushin \cite{dobrushinPrescribingSystemRandom1970a}]
  If \(\state\) is a set of Borel probability measures on a locally compact metric space \((X,d)\) such that for some \(x_0\in X\):
  \[\lim_{r\to\infty} \sup\{\int_{x: d(x,x_0)>r} d(x,x_0) d\mu(x): \mu\in \state\}=0,\]
  then the Monge-Kantorovich metric \(\mathrm{mk}_{\lip}\) metrizes the weak* topology on \(\state\).
\end{theorem}
Quantum compact metric spaces are required to satisfy the condition that the Monge-Kantorovich metric metrizes the weak* topology on the state space, or equivalently that the unit ball of the Lipschitz seminorm is totally bounded for the quotient norm on \(\sa(\A)/\mathbb{R}1\).
\begin{definition}[quantum compact metric spaces \cite{Rieffel1999}]
  A quantum compact metric space is a unital Lipschitz pair \((\A, \lip)\) such that the Monge-Kantorovich metric \(\mathrm{mk}_{\lip}\) metrizes the weak* topology on the state space \(\mathcal{S}(\A)\).
\end{definition}
It was shown in \cite{Rieffel1999} that this definition is equivalent to the total boundedness of the unit ball of the Lipschitz seminorm for the quotient norm on \(\sa(\A)/\mathbb{R}1\).

Similarly, in \cite{latremoliereQuantumLocallyCompact} locally compact quantum metric spaces were required to satisfy the condition that the Monge-Kantorovich metric metrizes the weak* topology on certain subsets of the state space that generalize the above Dobrushin tightness condition. To this end, the notion of topography was introduced to play the role of specifying a notion \(\lim_{r\to\infty}\ast=0\).
\begin{definition}[Lipschitz triple]
  A Lipschitz triple is a triple \((\A, \M, \lip)\) where:
  \begin{itemize}
    \item \((\A,\lip)\) is a Lipschitz pair,
    \item \(\M\) is an abelian C\(^*\)-subalgebra of \(\A\), called the \emph{topography}, such that \(\M\) contains an approximate unit for \(\A\),
  \end{itemize}
\end{definition}
The topography \(\M\) plays the role of specifying which elements of \(\A\) vanish at infinity, and allows one to define local states.
\begin{definition}[Local states \cite{latremoliereQuantumLocallyCompact}]
  We introduce the following:
  \begin{itemize}
    \item A state \(\mu\in \state(\A)\) is a \emph{local state} supported in \(K\subset \sigma(\M)\) compact if \(\mu(\chi_K)=1\). If this condition is satisfied we shall say that \(\mu\) is in \(\state(\A|K)\).
    \item The set of \emph{local states}  is denoted by \(\state(\A|\M)\): \[\state(\A|\M)=\bigcup_{K\subset \sigma(\M) \text{ compact}} \state(\A|K).\]
  \end{itemize}
\end{definition}
We define tame subsets of the state space \(\state(\A)\) that generalize the Dobrushin tightness condition.
\begin{definition}[Tame subsets of the state space \cite{latremoliereQuantumLocallyCompact}]
  A subset \(\mathcal{S}\subseteq \state(\A)\) is \emph{tame} if for some (hence any) local state \(\mu\in \mathcal{S}(\A|\M)\) we have:
  \[\lim_{K\in \mathcal{K}(\sigma(\M))} \sup\{|\varphi(a-\chi_K a \chi_K)|:\, \varphi\in \state, a\in \sa(\A), \lip(a)\leq 1, \mu(a)=0 \}=0.\]
\end{definition}
With these notions we can now define locally compact quantum metric spaces.
\begin{definition}[locally compact quantum metric spaces \cite{latremoliereQuantumLocallyCompact}]
  A Lipschitz triple \((\A, \M, \lip)\) is called a \emph{locally compact quantum metric space} if the Monge-Kantorovich metric \(\mathrm{mk}_{\lip}\) metrizes the weak* topology on any tame subset of the state space \(\mathcal{S}(\A)\).
\end{definition}
\begin{remark}
  Note that for any compact \(K\) in \(\sigma(\M)\), the set of local states \(\state(\A|K)\) is tame. Thus the Monge-Kantorovich metric metrizes the weak* topology on \(\state(\A|K)\), in particular:
  \[
    \diam(\state(\A|K),\mathrm{mk}_{\lip})<\infty.
  \]
  Lipschitz triples satisfying this condition are called \emph{regular}.
\end{remark}
Generalizing the characterization of compact quantum metric spaces in \cite{Rieffel1999}, \cite{latremoliereQuantumLocallyCompact} provides a characterization of locally compact quantum metric spaces in terms of total boundedness of certain sets.
\begin{theorem}[See \cite{latremoliereQuantumLocallyCompact}] \label{locallycompactqmscharacterization}
  A Lipschitz triple \((\A, \M, \lip)\) is a locally compact quantum metric space if and only if one of the following equivalent conditions holds:
  \begin{itemize}
    \item For all local states \(\mu\in \state(\A|\M)\) and for all compactly supported \(s,t\in\M\), the set:
          \begin{equation*}
            \{ sat \mid a\in \sa(\uu\A), \lip(a) \leq 1, \mu(a)=0 \}
          \end{equation*}
          is totally bounded for the norm topology of \(\A\).
    \item There exists a local state \(\mu\in \state(\A|\M)\) such that for all compactly supported \(s,t\in\M\), the set:
          \begin{equation*}
            \{ sat \mid a\in \sa(\uu\A), \lip(a) \leq 1, \mu(a)=0 \}
          \end{equation*}
          is totally bounded for the norm topology of \(\A\).
    \item (only if \(\A\) is separable) There exists a local state \(\mu\in \state(\A|\M)\) and a strictly positive element \(h\in\M\) such that the set:
          \begin{equation*}
            \{ hah \mid a\in \sa(\uu\A), \lip(a) \leq 1, \mu(a)=0 \}
          \end{equation*}
          is totally bounded for the norm topology of \(\A\).
  \end{itemize}
\end{theorem}
\begin{remark}
  In particular all the Lipschitz triples in this paper will be assumed to be separable (since we will be working with proper quantum metric spaces, see below).
\end{remark}
Recall that in the classical case, Lipschitz seminorms satisfy the Leibniz condition:
\[\lip(fg)\leq\|f\|\lip(g)+\lip(f)\|g\|.\]
This condition has been considered in the quantum case as well.
\begin{definition}[Leibniz Lipschitz pair \cite{latremoliereTopographicGromovHausdorffQuantum2014}]
  A Leibniz Lipschitz pair is a Lipschitz pair \((\A, \lip)\) such that:
  \begin{itemize}
    \item the domain \(\dom(\lip)\) is a Jordan-Lie subalgebra of \(\sa(\uu\A)\) i.e. it is closed under the Jordan bracket \(a\circ b=\frac{ab+ba}{2}\) and the Lie bracket \(\{a,b\}=\frac{ab-ba}{2i}\),
    \item for all \(a,b\in \dom(\lip)\), we have:
          \begin{align*}
            \lip(a\circ b) & \leq\|a\|\lip(b)+\lip(a)\|b\|, \\
            \lip(\{a,b\})  & \leq\|a\|\lip(b)+\lip(a)\|b\|.
          \end{align*}
  \end{itemize}
\end{definition}
As the spaces considered in coarse geometry are proper metric spaces, we will need a definition of properness for quantum metric spaces.

Denote by \(\mathfrak{Loc}(\A,\ast)\) the set of local elements:
\[\mathfrak{Loc}(\A,\ast)=\bigcup_{K\subset \sigma(\M) \text{ compact}} \sa(\chi_K \A \chi_K)\cap \dom(\lip).\]
\begin{definition}[locally compact proper quantum metric spaces\cite{latremoliereQuantumLocallyCompact}]
  A locally compact quantum metric space \((\A, \M, \lip)\) is said to be \emph{proper} if
  \begin{itemize}
    \item \(\A\) is separable,
    \item \((\A,\lip)\) is a Leibniz Lipschitz pair,
    \item \(\lip\) is lower semi-continuous with respect to the norm topology on \(\mathfrak{sa}(\A)\),
    \item \(\mathfrak{Loc}(\A,\ast)\cap \dom(\lip)\) is norm dense in \(\dom(\lip)\),
    \item \(\mathfrak{Loc}(\A,\ast)\cap \dom(\lip)\cap \M\) is norm dense in \(\M\),
    \item   \((\sigma(\M), \mathrm{mk}_{\lip|_{\M}})\) induced by \((\M, \lip|_{\M})\) is a proper metric space, i.e., all closed balls are compact. (note that \(\lip|_{\M}\) is densely defined on \(\M\) by the conditions above).
  \end{itemize}
\end{definition}
A more restrictive, easier to deal with notion of properness, called strong properness, was also introduced in \cite{latremoliereQuantumLocallyCompact}.
\begin{definition}[strongly proper quantum metric spaces \cite{latremoliereQuantumLocallyCompact}]\label{stronglyproperqms}
  A locally compact quantum metric space \((\A, \M, \lip)\) is said to be \emph{strongly proper} if:
  \begin{itemize}
    \item \(\lip\) can be extended to a dense subset of \(\A\), such that:
          \[\lip(ab)\leq \|a\|\lip(b)+\lip(a)\|b\|.\]
    \item \(\A\) is separable and \(\lip\) is lower semi-continuous,
    \item \(\{m\in\sa(\M)\mid \lip(m)<\infty\},\)
    \item there exists an approximate unit \((e_n)_{n\in\mathbb{N}}\) in \(\M\) for \(\A\) such that for all \(n\in\mathbb{N}\), we have \(\|e_n\|\leq 1\), \(e_n\in \chi_{K_n}\A\chi_{K_n}\) for some compact \(K_n\subset \sigma(\M)\), and \(\lim_{n\to\infty}\lip(e_n) = 0\).
  \end{itemize}
\end{definition}
\subsection{Noncommutative coarse geometry}
\begin{definition}[noncommutative coarse geometry \cite{banerjeeNoncommutativeCoarse}]
  A \emph{noncommutative coarse space} is a pair \((\A, \overline{\A})\) where \(\A\) is a \(C^*\)-algebra and \(\overline{\A}\) is a unital \(C^*\)-subalgebra of the multiplier algebra \(\multiplier(\A)\) containing \(\A\).
\end{definition}
\begin{remark}
  The datum above is equivalent to the datum of a unital \(C^*\)-algebra \(\overline{\A}\) of which \(\A\) is an essential ideal.
\end{remark}

In the above, the algebra \(\overline{\A}\) plays the role of the algebra of ``slowly oscillating functions'' and generalizes the Higson compactification in the commutative case. The pair encodes the large-scale (coarse) geometry of the noncommutative space. We recall below the precise statements in the motivating classical case, we refer to \cite{roeLecturesCoarseGeometry2003} and \cite{banerjeeNoncommutativeCoarse} for more details.

\begin{definition}[Higson compactification]
  For \(X\) a proper coarse space:
  \begin{itemize}
    \item  The algebra of slowly oscillating functions is given by:
          \[C^h(X)=\{f\in C_b(X): \mathbf{d}f \text{ vanishes at infinity on controlled sets} \}.\]
          where \(\mathbf{d}f(x,y)=f(x)-f(y)\) for all \((x,y)\in X\times X\).
    \item The spectrum of \(C^h(X)\) is called the Higson compactification of \(X\) and is denoted by \(hX\).
    \item In the case of proper metric spaces \((X,d)\), this definition can be made more explicit, and one has:
          \begin{align*}C^h(X)=\{f\in C_b(X): & \forall R>0, \forall \varepsilon>0, \exists K\subset X \text{ compact},\\& \forall x,y\notin K, d(x,y)\leq R \Rightarrow |f(x)-f(y)|\leq \varepsilon\}.\end{align*}
  \end{itemize}
\end{definition}

Furthermore if one starts from a compactification \(\overline{X}\) of a locally compact space \(X\), one can define a coarse structure on \(X\), denoted by \(tX\), by declaring a subset \(E\subset X\times X\) to be controlled if and only if nets in \(E\), \((x_i,y_i)\) converge to \(\partial X\) only mutually, i.e. if \(x_i\to \omega\in \partial X\) then \(y_i\to \omega\) as well.

It is then shown in \cite{roeLecturesCoarseGeometry2003},\cite{banerjeeNoncommutativeCoarse} that these two constructions are inverse of each other in the case of proper metric spaces, i.e. starting from a proper metric space \(X\), the coarse structure induced by the Higson compactification coincides with the coarse structure coming from the metric, and starting from a compactification \(\overline{X}\) of a locally compact space \(X\), the Higson compactification of the coarse structure induced by \(\overline{X}\) coincides with \(\overline{X}\). This motivates the noncommutative definition relying on fixing an analogue to the algebra of slowly oscillating functions.

Although the two constructions are in some sense equivalent to each other, the coarse structures obtained from considering a compactified topological space (with respect to some natural compactification, for instance if one considers a manifold with boundary), and those obtained from the metric coarse structure can be very different. For instance, already for \(\mathbb{R}\) with its usual metric, the Higson compactification associated with the metric coarse structure is not second countable. If one has a locally compact Riemannian manifold \(M\) that can be viewed as the interior of a compact manifold with boundary \(\overline{M}\), then the coarse structure induced by the compactification \(\overline{M}\) is in general different from the metric coarse structure. The interplay between the two can lead to interesting results, in index theory in particular, see for instance \cite{HigsonRoe}.
\begin{example}[The Euclidean plane]
  Consider \(\mathbb{R}^2\), then one obtains two coarse structures:
  \begin{itemize}
    \item The metric coarse structure coming from the usual Euclidean metric.
    \item The coarse structure induced by the compactification \(\mathbb{R}^2\cong \mathbb{D}^2\subset \overline{\mathbb{D}}^2\). Where \(\overline{\mathbb{D}}^2\) (resp. \(\mathbb{D}^2\)) is the closed (resp. open) unit disk.
  \end{itemize}
  In \cite{banerjeeNoncommutativeCoarse} two constructions of noncommutative coarse structures are considered:
  \begin{itemize}
    \item One obtained from continuously square-integrable, cocompact actions of locally compact groups on \(C^*\)-algebras, which generalizes the fact that a cocompact action of a locally compact group \(G\) on a locally compact space \(X\) induces a coarse structure on \(X\), that makes \(X\) coarsely equivalent to \(G\). This construction is reminiscent of the metric coarse structure on \(\mathbb{R}^2\), if one considers the action of \(\mathbb{R}^2\) (or \(\mathbb{Z}^2\)) on \(\mathbb{R}^2\) by translations, but it starts from an action of a group, rather than from a metric.
    \item One obtained from Rieffel deformation of \(X,\overline{X}\), again equipped with a continuous action of \(G\). The main example being the Moyal plane, obtained from the Rieffel deformation of \(\mathbb{R}^2\), along with the Rieffel deformation of \(\overline{\mathbb{D}}^2\). This construction is thus reminiscent of the second coarse structure above.
  \end{itemize}
\end{example}
In this paper we take a different approach since we start from a quantum metric space, in particular we are motivated by the case of the Moyal plane, where we can start from the quantum metric structure constructed in \cite{latremoliereQuantumLocallyCompact}, rather than from a Rieffel deformation, this is reminiscient of the metric coarse structure on \(\mathbb{R}^2\). Comparing with the ``topological coarse struture'' on the Moyal plane considered in \cite{banerjeeNoncommutativeCoarse} in order to conceptualize some index computations on the Moyal plane would be an interesting question for future research.


\subsection{W*-quantum metric spaces}
Kuperberg and Weaver in \cite{kuperbergNeumannAlgebraApproach2012} introduced a notion of quantum metric spaces in the setting of von Neumann algebras, inspired by the notion of finite propagation operators. We recall here some of their definitions.
\begin{definition}[W*-filtration \cite{kuperbergNeumannAlgebraApproach2012}]
  A \emph{W*-filtration} of \(\B(\HH)\) is a family \(\V=(\V_r)_{r\geq 0}\) of dual operator systems (i.e. weak* closed self-adjoint unital subspaces) in \(\B(\HH)\) such that:
  \begin{enumerate}
    \item \(\forall r,s\geq 0, \ \V_r \cdot \V_s \subseteq \V_{r+s}\),
    \item \(\forall r\geq 0, \ \V_r = \bigcap_{s>r} \V_s\).
  \end{enumerate}
\end{definition}

\begin{definition}[W*-quantum metric space \cite{kuperbergNeumannAlgebraApproach2012}]
  A W*-quantum pseudometric space is a von Neumann algebra \(\mathcal{M}\subset \B(\HH)\) such that \(\B(\HH)\) has a W*-filtration \(\V=(\V_r)_{r\geq 0}\) such that \(\mathcal{M}' \subseteq \V_0\).

  If in addition \(\mathcal{M}' = \V_0\), then we say that \((\mathcal{M}, \V)\) is a W*-quantum metric space.
\end{definition}

The intuition is that the elements of \(\V_r\) are operators of propagation at most \(r\). In the commutative discrete case, if \(X\) is a discrete metric space with metric \(d\), and \(\HH=l^2(X)\), then one can define:
\[\V_r^d=\{T\in \B(l^2(X)): d(x,y)>r \Rightarrow \langle T\delta_x, \delta_y\rangle=0\}.\]
Then \((l^{\infty}(X), \V^d)\) is a W*-quantum metric space. \(\V_0\) are operator of propagation \(0\), i.e. multiplication operators.

If \(\mathcal{M}\) is represented on two Hilbert spaces \(\HH_1\) and \(\HH_2\), then \cite{kuperbergNeumannAlgebraApproach2012} shows that there is 1-1 correspondence between W*-filtrations on \(\B(\HH_1)\) and \(\B(\HH_2)\) making \((\mathcal{M}, \V^{(1)})\) and \((\mathcal{M}, \V^{(2)})\) W*-quantum metric spaces. Thus the notion of W*-quantum metric space does not depend on the representation.

Furthermore, \cite{kuperbergNeumannAlgebraApproach2012} defines the following notion of reflexivity for W*-quantum metric spaces.
\begin{definition}[Reflexive W*-quantum metric space \cite{kuperbergNeumannAlgebraApproach2012}]
  A W*-quantum metric space \((\mathcal{M}, \V)\) is said to be reflexive if for all \(r\geq 0\), \(\V_r\) is reflexive as a dual operator space, i.e. :
  \[\V_r=\{B\in \B(\HH)\mid P\V_r Q=0\implies  P B Q=0 \}\]
\end{definition}
Noting that projections in \(\mathcal{M}\) can be endowed with a notion of distance coming from the W*-filtration, the definition is useful in that it allows one to recover the W*-filtration from the distance between projections of \(\mathcal{M}\).

\section{Roe algebras: Propagation and Commutation seminorm quantum metric spaces}
\label{sectionRoeAlgebras}
\begin{notation}[Setting]
  We set the following notations throughout this section:
  \begin{itemize}
    \item \((\A,\M,\lip)\) a proper locally compact quantum metric space.
    \item \(\BB\) a subset of \(\A\).
    \item \(\rho:\A \to \B(\HH)\) a faithful nondegenerate representation of \(\A\) on a Hilbert space \(\HH\).
    \item As an abuse of notation we shall write \(a\) for \(\rho(a)\) if \(a\) is an element of \(\A\) and whenever it is deemed that this will not cause confusion.
  \end{itemize}
\end{notation}
\begin{notation}[Distance on \(\sigma(\M)\)]
  If \(\M=C_0(\sigma(\M))\) is equipped with the restriction of the Lipschitz semi-norm \(\lip|_\M\), then \(\sigma(\M)\) is a metric space with the usual distance that we denote:
  \begin{align*}
    d(x,y):=\mathrm{mk}_{\lip|_{\M}}(\delta_x,\delta_y)=\sup\{|f(x)-f(y)|:\, f\in \M\cap \dom(\lip), \lip(f)\leq 1\}.
  \end{align*}
  In particular since \((\A,\M,\lip)\) is proper, \((\sigma(\M),d)\) is a proper metric space.

  More generally if \(\BB\) is a commutative subalgebra of \(\A\), we denote by \(d_{\BB}\) the distance on \(\sigma(\BB)\) defined by the restriction of \(\lip\) to \(\BB\).
\end{notation}

\subsection{Spectral Roe algebras}
For a locally compact quantum metric space \((\A,\M,\lip)\) and a representation \(\rho:\A\to\B(\HH)\), we shall define a notion of spectral Roe algebra.
\subsubsection{Definition and first properties}
\begin{definition}[Spectral propagation]
  Let \(T\in \B(\HH)\).
  We say that \(T\) has spectral propagation less than or equal to \(R\)  if:

  For all \(a\in \sa(\uu\A)\) such that \(\lip(a)\leq 1\):
  \[
    \chi_{(-\infty,0]}(a)T\chi_{(R,+\infty)}(a)=0
  \]

  We define \(\propagationS(T)\) to be the smallest such \(R\):
  \[
    \propagationS(T)=\inf\{R \in \mathbb{R}_+:\,\, T \text{ has propagation less than or equal to } R\}
  \]

  Furthermore let,
  \begin{itemize}
    \item \(B^{\ast}\) denote the closure of the set of operators of finite propagation,
    \item \(B^{\ast}_{R}\) be the set of operators of propagation less than or equal to \(R\).
  \end{itemize}
\end{definition}
\begin{remark}
  By applying an affine transformation, \(T\) having spectral propagation less than or equal to \(R\) is equivalent to: \[
    a\in \sa(\uu\A), \lip(a)\leq C \implies\chi_{(-\infty,\alpha]}(a)T\chi_{(\beta,+\infty)}(a)=0 , \quad \text{whenever } \beta-\alpha> CR.
  \]
\end{remark}
We then have the following properties:
\begin{prop}\label{propagationproperties}
  For operators \(T,S\in \B(\HH)\) the following holds:
  \begin{enumerate}
    \item \label{addpropagation} \(\propagationS(T+S)\leq \max(\propagationS(T),\propagationS(S)),\)
    \item \label{multpropagation}  \(\propagationS(TS)\leq \propagationS(T)+\propagationS(S)\),
    \item \label{starpropagation} \(\propagationS(T^{\ast})=\propagationS(T)\).
  \end{enumerate}
  Furthermore for all \(R>0\):
  \begin{enumerate}[start=4]
    \item \label{weakstarpropagation}\(B^{\ast}_{R}\) is an operator system in \(\B(\HH)\) that is closed in the weak operator topology.
    \item \label{intersectionpropagation} \(B^{\ast}_{R}=\bigcap_{R'>R} B^{\ast}_{R'}\),
    \item \label{zeropropagation} \(B^{\ast}_{0}=\A'\): the elements of zero propagation are precisely the elements of the commutant of \(\A\).
  \end{enumerate}
\end{prop}
\begin{proof}
  \textbf{(\ref{addpropagation})} : Suppose that \(\propagationS(T)\leq R\) and \(\propagationS(S)\leq R\). Let \(a\in \sa(\uu\A)\) such that \(\lip(a)\leq 1\), then:
  \[
    \chi_{(-\infty,0]}(a)T\chi_{(R,+\infty)}(a)=0, \quad \chi_{(-\infty,0]}(a)S\chi_{(R,+\infty)}(a)=0,
  \]
  hence:
  \[
    \chi_{(-\infty,0]}(a)(T+S)\chi_{(R,+\infty)}(a)=0.
  \]

  \textbf{(\ref{multpropagation})} : Suppose that \(\propagationS(T)\leq R_1\) and \(\propagationS(S)\leq R_2\). Let \(a\in \sa(\uu\A)\) such that \(\lip(a)\leq 1\), then:
  \begin{align*}
    \chi_{(-\infty,0]}(a)TS\chi_{(R_1+R_2,+\infty)}(a) & =\chi_{(-\infty,0]}(a)T\chi_{(-\infty,R_1]}(a)S\chi_{(R_1+R_2,+\infty)}(a)=0 \\
  \end{align*}
  Since \(\chi_{(-\infty,0]}(a)T=\chi_{(-\infty,0]}(a)T\left(\chi_{(-\infty,R_1]}(a)+\chi_{(R_1,+\infty)}(a)\right)=\chi_{(-\infty,0]}(a)T\chi_{(-\infty,R_1]}(a)\) by the propagation condition on \(T\), and
  \(\chi_{(-\infty,R_1]}(a)S\chi_{(R_1+R_2,+\infty)}(a)=0\) by the propagation condition on \(S\).

  \textbf{(\ref{starpropagation})} : Indeed \(\chi_{(-\infty,0]}(a)T^{\ast}\chi_{(R,+\infty)}(a)=\chi_{(R,+\infty)}(a)T\chi_{(-\infty,0]}(a)\) and we can conclude by applying an affine transformation on \(a\).

  \textbf{(\ref{weakstarpropagation})} : By the above properties \(B^{\ast}_{R}\) is a unital operator system. Let \((T_i)_i\) be a net in \(B^{\ast}_{R}\) that converges to \(T\) in the weak* topology. For \(a\in \sa(\uu\A)\) such that \(\lip(a)\leq 1\), and \(\xi,\eta \in \HH\), we have:
  \begin{align*}
    0=\braket{\eta,\chi_{(-\infty,0]}(a)T_i\chi_{(R,+\infty)}(a)\xi}\xrightarrow[i]{}
    \braket{\eta,\chi_{(-\infty,0]}(a)T\chi_{(R,+\infty)}(a)\xi}.
  \end{align*}

  \textbf{(\ref{intersectionpropagation})}: Indeed if \(T\in \bigcap_{R'>R} B^{\ast}_{R'}\), then for all \(a\in \sa(\uu\A)\) such that \(\lip(a)\leq 1\) and for all \(\eta \in \HH\) and \(R'>R\):
  \begin{align*}
    \chi_{(-\infty,0]}(a)T\chi_{(R',+\infty)}(a)\eta=0
  \end{align*}
  Taking the limit \(R'\to R\), we obtain, since the Borel functional calculus is continuous from pointwise convergence to the strong operator topology and since \(\chi_{(R',+\infty)}(a)\) converges pointwise to \(\chi_{(R,+\infty)}(a)\) that:
  \begin{align*}
    \chi_{(-\infty,0]}(a)T\chi_{(R,+\infty)}(a)\eta=0.
  \end{align*}

  \textbf{(\ref{zeropropagation})}: If \(T\in B^{\ast}_{0}\), then it follows from proposition \ref{fromSpectralToRelativeCommutant} below that for all \(a\in \sa(\A)\): \([T,a]=0\) (since in the notation of proposition \ref{fromSpectralToRelativeCommutant}, \(L^{\ast}(T)=0\)), hence \(T\in \A'\). The reverse inclusion is clear: \(\chi_{(-\infty,0]}(a)T\chi_{(R,+\infty)}(a)=\chi_{(-\infty,0]}(a)\chi_{(R,+\infty)}(a)T=0\).
\end{proof}

Recall that for \(\A\), we denote by \(\A^{\ast\ast}\) the enveloping von Neumann algebra, if \(\rho\) is a faithful nondegenerate representation of \(\A\) then  \(\A^{\ast\ast}\) can be identified with \(\rho(\A)''\) the \(W^{\ast}\)-algebra generated by  \(\rho(\A)\). From the above discussion, we get then the following as a corollary, this is essentially a restatement of the definitions.
\begin{theorem}{}\label{C*metricToW*metric}
  If \((\A,\M,\lip)\) is a locally compact quantum metric space, then the family \((B^{\ast}_R)_R\) defines a W*-quantum  metric on \(\A^{\ast\ast}\) \cite{kuperbergNeumannAlgebraApproach2012}.

  Furthermore this quantum metric is reflexive.
\end{theorem}
\begin{proof}[Proof of reflexivity]
  We need to show that if \(T\in \B(\HH)\) is such that if:
  \begin{align} \label{reflexivity1}
    PB_R^{\ast}Q=\{0\} \implies PTQ=0, \text{ for all projections } P,Q \in \A^{\ast\ast},
  \end{align}
  then \(T\) is in  \(B_R^{\ast}\).

  Suppose \(T\) satisfies \ref{reflexivity1}, and let \(a\in \sa(\uu\A)\) such that \(\lip(a)\leq 1\). If we denote by \(P=\chi_{(-\infty,0]}(a)\) and \(Q=\chi_{(R,+\infty)}(a)\), then by the definition of spectral propagation \(PB_R^{\ast}Q=\{0\}\), hence by \ref{reflexivity1}:
  \begin{align*}
    \chi_{(-\infty,0]}(a)T\chi_{(R,+\infty)}(a)=0.
  \end{align*}
  Thus \(T\in B_R^{\ast}\).
\end{proof}
Recall that an operator \(T\) is locally compact if for all \(a\in \A\): \(Ta \in \K(\HH)\) and \(aT\in \K(\HH)\).
\begin{definition}[Spectral Roe algebra]
  The Spectral Roe algebra associated with \((\A,\M,\lip)\) and \(\rho\) is the \(C^\ast\)-algebra generated by locally compact finite propagation operators.
  \begin{align*}
    C^\ast_{\spec}(\A)=\overline{\{T\in \B(\HH):\,T \textrm{ has finite spectral propagation and T is locally compact} \}}.
  \end{align*}
  Furthermore we denote by \(D^\ast_{\spec}(\A)\)
  the C\(^\ast\) algebra generated by the operators of finite spectral propagation that commute with \(\A\) modulo compacts.
\end{definition}


\subsection{Relative commutant Roe algebra}
\begin{definition}[Commutant seminorm]
  For \(T\in \B(\HH)\), we define the commutant seminorm of \(T\) with respect to \((\A,\M,\lip)\) as:
  \begin{align*}
    \lip^{\ast}(T)=\sup\{\|[T,a]\|:\, a\in \sa(\A), \lip(a)\leq 1\}.
  \end{align*}
\end{definition}
\begin{remark}
  We then have
  \[
    \|[T,a]\|\leq L^{\ast}(T)L(a) \quad \text{ for all } a\in \sa(\uu\A).
  \]
\end{remark}
\begin{prop}
  For operators \(T,S\in \B(\HH)\) and a scalar \(\lambda\) the following holds:
  \begin{enumerate}
    \item \(\lip^{\ast}(\lambda T)=|\lambda|\lip^{\ast}  (T),\) \label{liphom}
    \item \(\lip^{\ast}(T+S)\leq \lip^{\ast}(T)+\lip^{\ast}(S),\) \label{liptriangle}
    \item \(\lip^{\ast}(TS)\leq \|T\|\lip^{\ast}(S)+\|S\|\lip^{\ast}(T)\), \label{liplip}
    \item \(\lip^{\ast}(T^{\ast})=\lip^{\ast}(T)\). \label{lipstar}
  \end{enumerate}
\end{prop}
\begin{proof}
  \ref{liphom},\ref{liptriangle} and \ref{lipstar} follow from the corresponding properties for the norm, and from the fact that the commutator is bilinear and that \([T^{\ast},a]=[T,a]^{\ast}\).

  \ref{liplip} follows from the fact that \([TS,a]=T[S,a]+[T,a]S\).
\end{proof}

\begin{prop}[From Spectral to Relative commutant]
  \label{fromSpectralToRelativeCommutant}
  Let \((\A,\M,\lip)\) be locally compact quantum metric space represented on \(\HH\), then for any \(T\in\B(\HH)\):\[
    \lip^{\ast}(T)\leq 3\propagationS(T)\|T\|\]
\end{prop}
\begin{proof}
  We show that for \(a\in \sa(\A)\) such that \(\lip(a)\leq 1\), and denoting \(R=\propagationS(T)\), we have \(\|[T,a]\|\leq 3R\|T\|\).

  Fix the following:
  \begin{itemize}
    \item \((\alpha_i)_i\) a sequence of real numbers such that: \(\alpha_{i+1}-\alpha_i=R\),
    \item \((\lambda_i)_i\) the medians of \(\alpha_i\) and \(\alpha_{i+1}\) : \(\lambda_i=\left(\alpha_i+\alpha_{i+1}\right)/2\),
    \item \((\chi_i)_i\) are the elements of \(\A^{\ast\ast}\) defined by : \(\chi_i=\chi_{[\alpha_i,\alpha_{i+1})}(a),\)
    \item \(a'\) is the element : $
            a'=\sum_i \lambda_i \chi_i $
  \end{itemize}
  We then have :
  \begin{align*}
    \|a-a'\|\leq R/2.
  \end{align*}
  Thus :
  \begin{align}\label{eq:aaprimeEstimate}
    \|[T,a]\|\leq \|[T,a']\|+ \|[T,a-a']\|\leq \|[T,a']\|+ R\|T\|.
  \end{align}
  By the definition of spectral propagation we have \(\chi_iT=\chi_{i}T\left(\chi_{i-1}+\chi_i+\chi_{i+1}\right)\),\\ \(T\chi_i=\left(\chi_{i-1}+\chi_i+\chi_{i+1}\right)T\chi_{i},\) thus :
  \begin{align*}
    [T,a'] & =\sum_i \lambda_i [T,\chi_i]=\sum_i \left(\lambda_i \chi_{i-1}T\chi_i-\lambda_i \chi_i T\chi_{i-1}\right)-\sum_i \left(\lambda_i \chi_{i}T\chi_{i+1}-\lambda_i \chi_{i+1}T\chi_i\right) \\
           & =\sum_i \left(\lambda_i-\lambda_{i-1}\right)\left( \chi_{i-1}T\chi_i- \chi_i T\chi_{i-1}\right)=R\sum_i \left( \chi_{i-1}T\chi_i- \chi_i T\chi_{i-1}\right).
  \end{align*}
  Since the terms in the sums \(\sum_i \chi_{i-1}T\chi_i\) and \(\sum_i \chi_i T\chi_{i-1}\) are orthogonal, we have :
  \begin{align}\label{eq:aprimeEstimate}
    \|[T,a']\|\leq 2R\|T\|.
  \end{align}
  Combining \eqref{eq:aaprimeEstimate} and \eqref{eq:aprimeEstimate} we obtain the desired result : \(\|[T,a]\|\leq 3R\|T\|\).

\end{proof}

\begin{definition}[Relative commutant Roe algebra]
  The relative commutant Roe algebra associated with \((\A,\M,\lip)\) and \(\rho\) is the \(C^\ast\)-algebra generated by locally compact operators of finite commutant seminorm.
  \begin{align*}
    C^\ast_{\comm}(\A)=\overline{\{T\in \B(\HH):\,T \textrm{ has finite commutant seminorm and } T \textrm{ is locally compact} \}}.
  \end{align*}
\end{definition}
Furthermore we denote by \(D^\ast_{\comm}(\A)\)
the C\(^\ast\) algebra generated by the operators of finite commutant seminorm that commute with \(\A\) modulo compacts.

\subsubsection{Finite asymptotic dimension : equality of the algebras}
In this part, as a motivation, we specialize to the classical context and recall the methods introduced in section 4 of \cite{spakulaRelativeCommutantPictures2019}, rephrasing them in our context to show how finite asymptotic dimension leads to equality between spectral Roe algebra and the relative commutant Roe algebra. In fact, the equality there is proven for a larger algebra, the so-called quasi-local algebra \(C^{\ast}_q(X)\), which contains the relative commutant Roe algebra. In general we have the inclusions:
\[
  C^{\ast}_{\spec}(X)\subset C^{\ast}_{\comm}(X) \subset C^{\ast}_q(X).
\]
where:
\begin{definition}[Quasi-local operators]
  An operator \(T\in \B(\HH)\) is quasi-local if for every \(\varepsilon>0\) there exists \(R>0\) such that for all \(f,g\in C_b(X)\) with \(d(\supp(f),\supp(g))>R\):
  \[
    \|fTg\|\leq \varepsilon \|f\|\|g\|.
  \]

  The quasi-local algebra \(C^{\ast}_q(X)\) is the C*-algebra of locally compact quasi-local operators.
\end{definition}

\begin{definition}[Asymptotic dimension]
  For a proper metric space \(X\), we say that \(\asdim\,X\leq n\) if and only if for every \(R>0\) there exists \(D=D(R)<\infty\) and a cover \(\mathcal{U}\) of \(X\) consisting of sets of diameter \(\le D\) such that every ball of radius \(R\) meets at most \(n+1\) members of \(\mathcal{U}\) (i.e. the \(R\)-multiplicity of \(\mathcal{U}\) is \(\le n+1\)).
\end{definition}
\begin{theorem}[A variation of the results of \cite{spakulaRelativeCommutantPictures2019}]\label{asdimrelativecommutantroetheorem}
  Let \((X,d)\) be a proper metric space such that \(\asdim\,X\leq n\). Then for any representation \(\rho:C_0(X)\to \B(\HH)\), we have:
  \[
    C^\ast_{\spec}(X)=C^\ast_{\comm}(X).
  \]
  (In fact more generally \(C^\ast_{\spec}(X)=C^\ast_{q}(X)\).)
\end{theorem}
The proof of the theorem relies on the following version of corollary 4.3 of \cite{spakulaRelativeCommutantPictures2019}.
\begin{lemma}
  \label{commutationlemmaSpakula}
  Let \((X,d)\) be a proper metric space, \(T\in \B(\HH)\) and  \((e_j)_j\) a family of positive contractions in \(C_b(X)\)  with \(R\)-disjoint supports, and let \(e:=\sum_{j}e_j\), then:
  \[
    \|eTe-\sum_j e_jTe_j\|\leq \frac{2}{R}\lip^{\ast}(T).
  \]
\end{lemma}
The following cutting procedure allows us then to approximate an operator with finite commutant seminorm by a sum of operators with finite propagation. Since we will use it in the stronger control case of finite Assouad-Nagata dimension, we state it as a lemma.

\begin{lemma}
  Let \((X,d)\) be a proper metric space, \(\varepsilon>0\), \(R>0\) and \(T\in \B(\HH)\).
  Furthermore let \(\mathcal{U}\) be a cover of uniformly bounded diameter \(D\) and \(R\)-multiplicity less than or equal \(n+1\).

  Then there exists  \(T_{ii'}\), for \(0\leq i,i'\leq n\) such that :
  \begin{itemize}
    \item The following approximation inequality holds: \[\|T- \sum_{i,i'} T_{ii'}\|\leq 16(n+1)^2 \frac{\lip^{\ast}(T)}{R}.\]
    \item The \(T_{ii'}\)'s have propagation less than or equal \(R/2+D\).
    \item More precisely, each \(T_{ii'}\) is of the form \(\sum_k e_k gTg'e_k\) for some contractions \(g,g'\) in \(C_b(X)\) and some family of \(4/R\)-Lipschitz positive contractions \((e_k)_k\) in \(C_b(X)\)  with disjoint supports, such that the support of each \(e_k\) is contained in a set \(B_{3R/8}(U)\) where \(U\in\mathcal{U}\).
  \end{itemize}
\end{lemma}
\begin{proof}
  Let \(\mathcal{U}^{(0)},\dots, \mathcal{U}^{(n)}\) be a partition of \(\mathcal{U}\) such that each \(\mathcal{U}^{(i)}\) is \(R\)-disjoint. We write \(\mathcal{U}^{(i)}=\{U_{j}^{(i)}\}_j\) and let \((e_j^{(i)})_{j,i}\) be a partition of unity subordinate to \(\mathcal{U}\), such that each \(e_j^{(i)}\) is a \(4/R\)-Lipschitz positive contraction such that the support of \(e_j^{(i)}\) is contained in \(B_{R/4}(U_j^{(i)})\). Then the supports of \(\{e_j^{(i)}\}_j\) are \((R/2)\)-disjoint for each fixed \(i\).

  For each \(i\in\{0,\dots,n\}\) define \(e^{(i)}=\sum_j e_j^{(i)}\). Then:

  \begin{align*}
    T=\sum_{ii'} e^{(i)}Te^{(i')}.
  \end{align*}
  For the case \(i=i'\), we have by lemma \ref{commutationlemmaSpakula} above:
  \begin{align*}
    \|e^{(i)}Te^{(i)}- \sum_j e_j^{(i)}Te_j^{(i)}\|
     & \leq 4\frac{\lip^{\ast}(T)}{R}.
  \end{align*}
  note that since the support of each \(e_j^{(i)}\) is contained in \(B_{R/4}(U_j^{(i)})\) , \(T_{ii}\) has propagation less than or equal to \(R/4+D\).

  For the case \(i\neq i'\), we will use a similar reasoning as above, applied to the family \((B_{3R/8}(U_{j}^{(i)})\cap B_{3R/8}(U_{j'}^{(i')}))_{j,j'}\) and using the fact that it is an \(R/4\)-disjoint family.

  For each \(j\), let \(p^{(i)}_j\) be the characteristic function associated with the support of \(e^{(i)}_j\), and let \(p^{(i)}:=\sum_j p^{(i)}_j\), then:
  \begin{align*}
    e^{(i)}Te^{(i')} & = p^{(i)}e^{(i)}Te^{(i')}p^{(i')}
  \end{align*}
  Thus, using the definition of \(L^{\ast}(T)\) and the fact that \(\lip(e^{(i)})\leq 4/R\), we have:
  \begin{align}\label{commutineq1}
    \| e^{(i)}Te^{(i')}-p^{(i)}e^{(i')}Te^{(i)}p^{(i')}\|\leq\frac{8}{R}\lip^{\ast}(T).
  \end{align}
  For each \(j\), \(j'\), there is \(f_{j,j'}\), a \(8/R\)-Lipschitz positive contraction that is 1 on  \(B_{R/4}(U_{j}^{(i)})\cap B_{R/4}(U_{j'}^{(i')})\) and supported in \(B_{3R/8}(U_{j}^{(i)})\cap B_{3R/8}(U_{j'}^{(i')})\). If we let \(f=\sum_{j,j'}f_{j,j'}\) then:
  \begin{align*}
    fp^{(i)}e^{(i')} & =p^{(i)}e^{(i')}   \\
    e^{(i)}p^{(i')}f & = e^{(i)}p^{(i')},
  \end{align*}
  and by lemma \ref{commutationlemmaSpakula}, since the supports of \(f_{j,j'}\) are \(R/4\)-disjoint :
  \begin{align}\label{commutineq2}
    \|fp^{(i)}e^{(i')}Te^{(i)}p^{(i')}f-\sum_{j,j'}f_{j,j'}p^{(i)}e^{(i')}Te^{(i)}p^{(i')}f_{j,j'}\|\leq \frac{8}{R}\lip^{\ast}(T).
  \end{align}
  Combining \ref{commutineq1} with \ref{commutineq2}, it follows that :
  \begin{align*}
    \| e^{(i)}Te^{(i')}-\sum_{j,j'}f_{j,j'}e^{(i)}\hat{e}^{(i')}T\hat{e}^{(i)}e^{(i')}f_{j,j'}\|\leq \frac{16}{R}\lip^{\ast}(T).
  \end{align*}
  Setting \(\sum_{j,j'}f_{j,j'}e^{(i)}\hat{e}^{(i')}T\hat{e}^{(i)}e^{(i')}f_{j,j'}=:T_{ii'}\), noting that it has propagation less than or equal \(3R/8+D\), and adding the \((n+1)^2\) terms corresponding to \(i\) and \(i'\) concludes the proof.

\end{proof}

The above lemma allows us to conclude.
\begin{proof}[Proof of theorem \ref{asdimrelativecommutantroetheorem}]
  Let \(T\in C_{\comm}^{\ast}(X)\) of finite \(\lip^{\ast}\)-seminorm, and \(\varepsilon>0\). By the lemma above, if we choose \(R> 16(n+1)^2 \frac{\lip^{\ast}(T)}{\varepsilon}\), and a decomposition of \(X\) into \(n+1\)  collections of sets, each collection being \(R\)-disjoint, then by lemma \ref{commutationlemmaSpakula} there exists \(T_{ii'}\), for \(0\leq i,i'\leq n\) such that :
  \begin{itemize}
    \item The following approximation inequality holds: \[\|T- \sum_{i,i'} T_{ii'}\|\leq \varepsilon.\]
    \item  \(\sum_{i,i'}T_{ii'}\) has propagation less than or equal \(R/2+D\).
    \item \(\sum_{i,i'}T_{ii'}\) is locally compact.
  \end{itemize}
  It follows by taking \(\varepsilon\to 0\) that \(T\) is in \(C_{\spec}^{\ast}(X)\).
\end{proof}

\subsubsection{Assouad-Nagata dimension : controlling propagation from commutation semi-norm}
First we recall the definition of Assouad-Nagata dimension, stated in similar terms as in the definition of asymptotic dimension. The difference is that the diameters of the sets partitioniong the space can be controlled.
\begin{definition}[Assouad-Nagata dimension]
  For a proper metric space \(X\), we say that \(\dim_{AN}\,X\leq n\) if and only if there exists a constant \(C>0\) such that for every \(R>0\) there exists a cover \(\mathcal{U}\) of \(X\) consisting of sets of diameter less than or equal \(CR\) such that every ball of radius \(R\) meets at most \(n+1\) members of \(\mathcal{U}\) (i.e. the \(R\)-multiplicity of \(\mathcal{U}\) is \(\le n+1\)).
\end{definition}
In this case we can see, as follows, that the commutation semi-norm controls, up to perturbation, the propagation. This can be seen as a kind a weaker converse to the fact we saw earlier that \(\lip^{\ast}(T)\leq 3\propagationS(T)\|T\|.\)
\begin{theorem}\label{assouadnagatarelativecommutantroetheorem}
  Let \((X,d)\) be a proper metric space such that \(\dim_{AN}\,X\leq n\). Then there exists a constant \(C>0\) (depending only on \(X\)) such that for any operator \(T\in \B(\HH)\) and \(\alpha>0\), there exists \(T'\in \B(\HH)\) such that:
  \begin{align*}
    \|T-T'\|          & < \alpha,                             \\
    \propagationS(T') & \leq C \frac{\lip^{\ast}(T)}{\alpha}.
  \end{align*}
\end{theorem}
\begin{proof}
  For \(T\in \B(\HH)\) and \(\alpha>0\). Following the same reasoning as in the proof of theorem \ref{asdimrelativecommutantroetheorem}, taking \(R=\frac{L^{\ast}(T)}{\alpha}16(n+1)^2\) and applying lemma \ref{commutationlemmaSpakula} with \(R\)-disjoint families of sets of diameter \(\le C'R\), and taking \(C=16(n+1)^2(C'+1/2)\) we conclude.
\end{proof}
\begin{remark}
  In particular by taking \(\alpha=\|T\|\), we obtain \(T'\) such that:
  \begin{align*}
    \|T-T'\|                & < \|T\|,               \\
    \|T'\|\propagationS(T') & \leq 2C\lip^{\ast}(T).
  \end{align*}
\end{remark}

\subsection{Spectral Roe algebras relative to a subalgebra and topographic Roe algebra}

In this section, we define Roe algebras relative to a subset \(\BB\) of \(\A\). In this case, using the restriction of the lipschitz seminorm \(\rho\) to \(\BB\), one can define finite propagation operators as follows:
\begin{definition}[Spectral propagation relatively to \(\BB\)]
  Let \(T\in \B(\HH)\).
  We say that \(T\) has spectral propagation less than or equal to \(R\)  if:

  For all \(a\in \sa(\uu\A)\cap \BB\) such that \(\lip(a)\leq 1\):
  \[
    \chi_{(-\infty,0]}(a)T\chi_{(R,+\infty)}(a)=0
  \]

  We define \(\propagationS_{\BB}(T)\) to be the smallest such \(R\):
  \[
    \propagationS_{\BB}(T)=\inf\{R \in \mathbb{R}_+:\,\, T \text{ has propagation smaller than or equal to } R\}
  \]
\end{definition}

Properties from proposition \ref{propagationproperties} obviously still hold in this relative setting:
\begin{prop}\label{propagationproperties1}
  For operators \(T,S\in \BB(\HH)\) the following holds:
  \begin{enumerate}
    \item \label{addpropagationB} \(\propagationS_{\BB}(T+S)\leq \max(\propagationS_{\BB}(T),\propagationS_{\BB}(S)),\)
    \item \label{multpropagationB}  \(\propagationS_{\BB}(TS)\leq \propagationS_{\BB}(T)+\propagationS_{\BB}(S)\),
    \item  \label{starpropagationB} \(\propagationS_{\BB}(T^{\ast})=\propagationS_{\BB}(T)\).
  \end{enumerate}
  Furthermore for all \(R>0\):
  \begin{enumerate}[start=4]
    \item \label{weakstarpropagationB}\(B^{\ast}_{R}\) is an operator system in \(\B(\HH)\) that is closed in the weak operator topology.
    \item \label{intersectionpropagationB} \(B^{\ast}_{R}=\bigcap_{R'>R} B^{\ast}_{R'}\),
    \item \label{zeropropagationB} \(B^{\ast}_{0}=\BB'\): the elements of zero propagation are precisely the elements of the commutant of \(\BB\).
  \end{enumerate}
\end{prop}

This time the relative propagation defines a W*-filtration on \(\B(\HH)\) that leads to a W*-quantum metric on \(\BB^{\ast\ast}\).

\begin{definition}[Spectral Roe algebra relatively to \(\BB\)]
  The Spectral Roe algebra associated with \((\A,\M,\lip)\) and \(\rho\) is the \(C^\ast\)-algebra generated by locally compact operators of finite propagation with respect to \(\BB\).
  \begin{align*}
    C^\ast_{\spec,\BB}(\A)=\overline{\{T\in \B(\HH):\,\propagationS_{\BB}(T)<+\infty \textrm{ and T is locally compact} \}}.
  \end{align*}
\end{definition}

For instance, one could take \(\BB\) to be a commutative subalgebra \(\A\) containing \(\M\). Then \(\sigma(\M)\) is a quotient space of \(\sigma(\BB)\), by the map taking characters on \(\BB\) to their restriction to \(\M\):
\begin{align*}
  p: & \sigma(\BB)\to\sigma(\M) \\
     & \phi\mapsto \phi|_{\M}.
\end{align*}
We have then metric spaces \((\sigma(\M),d_\M)\) and \((\sigma(\BB),d_\BB)\) defined from the lipschitz seminorms \(\lip|_{\M}\) and \(\lip|_{\BB}\) respectively, they make \(p\) into a contractive map. Since \((\A,\M,\lip)\) is a proper locally compact quantum metric space, \((\M,\M,\lip|_{\M})\) defines a proper locally compact metric space and, taking a strictly positive element \(h\) in \(\M\) and a local state \(\mu\) such that :
\begin{equation*}
  \{ hah \mid a\in \sa(\uu\A), \lip(a) \leq 1, \mu(a)=0 \} \text{ is norm precompact,}
\end{equation*}
then :
\begin{equation*}
  \{ hbh \mid b\in \sa(\uu\BB), \lip(b) \leq 1, \mu(b)=0 \} \text{ is norm precompact.}
\end{equation*}
Now since \(\M\) contains an approximate unit for \(\BB\) (since \(\M\subset\BB\subset \A\)), \(h\) is strictly positive in \(\BB\) as well, thus \((\BB,\BB,\lip|_{\BB})\) is a commutative locally compact quantum metric space, in other words \((\sigma(\BB),d_\BB)\) is a locally compact metric space.

Now if \(E\) is a closed and bounded subset of \((\sigma(\BB),d_\BB)\), \(F=\overline{p(E)}\) is closed and bounded in \((\sigma(\M),d_\M)\) thus compact since \((\sigma(\M),d_\M)\) is proper. Thus the set \(h(x)\) (\(h\in \M\) being the strictly positive element above) for \(x\in F\) has a positive lower bound \(m>0\). Now \(E\) is a closed subset of :
\begin{align*}
  E\subset p^{-1}(F)=\{\varphi\in \sigma(\BB)\mid \varphi|_{\M}\in F\}\subset\{\varphi\in\sigma(\BB)\mid \varphi(h)\geq m\}
\end{align*}
Where we identify elements of \(\sigma(\BB)\) with characters on \(\BB\). The set \(\{\varphi\in\sigma(\BB)\mid \varphi(h)\geq m\}\) is weak*-compact since it is closed in unit ball of \(\BB^{\ast}\), thus \(E\) is compact.
Therefore \((\sigma(\BB),d_\BB)\) is a proper metric space. We have shown that:
\begin{prop}
  If \(\BB\) is a commutative C*-subalgebra of \(\A\) containing \(\M\), then the Monge-Kantorovich metric associated with \(\lip|_{\BB}\) makes \(\sigma(\BB)\) into a proper metric space.
\end{prop}
The following is then obvious:
\begin{prop}
  If \(\BB\) is a commutative C*-subalgebra of \(\A\) containing \(\M\), then the spectral Roe algebra associated with \((\A,\M,\lip)\) relatively to \(\BB\)  is isomorphic to the classical Roe algebra \(C^{\ast}(\sigma(\BB))\).
\end{prop}
In particular if we take \(\BB=\M\) then we call \(C^{\ast}(\sigma(\M))\) the topographic Roe algebra associated with \((\A,\M,\lip)\).

\section{Generalized Higson compactification}
\subsection{Spectral Higson compactification}
In this section we show that given a locally compact quantum metric space \((\A,\M,\lip)\) represented on a Hilbert space \(\HH\), we can define a noncommutative coarse structure on \(\A\) that generalizes the Higson compactification for metric spaces. This structure shall depend (at least a priori) on the choice of the representation \(\rho:\A\to \B(\HH)\). We will briefly give an alternative construction, using the relative commutant algebra as well, the study of the relationship between these two constructions, wether they agree or not, starting from specific (noncommutative) examples, is left for future work, and the main computations in the example section will be done using the spectral Higson compactification.

More generally, we consider a C\(^\ast\)-algebra \(\A\) represented faithfully and nondegenerately on a Hilbert space \(\HH\), together with a topography \(\M\subset \A\), along with a W*-quantum metric on \(\A^{\ast\ast}\), \(\V\subset\B(\HH)\). We shall call the triple \((\A,\M,\V)\) a topographic W*-quantum metric space, and we shall write \(\propagationS(T)\leq R\) for \(T\in\V_R\).

Note that since \(\rho:\A\to\B(\HH)\) is faithful and nondegenerate, we can identify the multiplier algebra \(\multiplier(\A)\) with a subalgebra of \(\B(\HH)\). We shall define the generalized Higson compactification as a unital subalgebra of \(\multiplier(\A)\).

We need to have a notion for vanishing at infinity, and this is what the topography \(\M\) is for.
\begin{example}[locally compact quantum metric space is a topographic W*-quantum metric space]
  If \((\A,\M,\lip)\) is a proper locally compact quantum metric space represented on \(\HH\), then \((\A,\M,(B^{\ast}_R))\) is a topographic W*-quantum metric space, where \((B^{\ast}_R)_R\) is  the filtration defined using the spectral propagation as in section \ref{sectionRoeAlgebras}.
\end{example}
\begin{notation}
  We fix \((f_{\alpha})_{\alpha}\) an approximate unit of \(\A\) made from elements of \(\M\).
\end{notation}
\begin{definition}[Generalized Higson compactification]\label{genHigsonCompact}Let \((\A,\M,\V)\) be a topographic W*-quantum metric space:
  \begin{itemize}
    \item We shall say that an operator \(S\) in \(\multiplier(\A)\) is slowly oscillating if for all \(R>0\) and all \(\varepsilon>0\), there is an \(\alpha_0\) such that if \(\alpha\geq \alpha_0 \), \(\|T\|\leq 1\) and \(T\in \V_R\) then \(\|(1-f_\alpha)[T,S]\|\leq \varepsilon\) and \(\|[T,S](1-f_\alpha)\|\leq \varepsilon\) .
    \item We define the Generalized Higson compactification of \((\A,\M,\V)\),  to be the set \(\overline{\A}^H\) of elements of \(\multiplier(\A)\) that are slowly oscillating.
  \end{itemize}
\end{definition}
\begin{lemma}
  An element \(S\) in \(\multiplier(\A)\) is slowly oscillating if and only if for all \(R>0\) and \(\varepsilon>0\) there is \(K\subset \sigma(\M)\) compact such that \(\|\chi_{K^\complement}[T,S]\|\leq\varepsilon\) and \(\|[T,S]\chi_{K^\complement}\|\leq\varepsilon\) for all \(T\) such that \(T\in \V_R\) and \(\|T\|\leq 1\).

  As a consequence definition \ref{genHigsonCompact} does not depend on the approximate unit \(\left(f_{\alpha}\right)\).
\end{lemma}
\begin{proof}
  Let \(f_\alpha\) be an approximate unit of \(\A\) made of elements in \(\M\) and \(S\in \multiplier(\A)\)  and suppose that \(S\) satisfies definition \ref{genHigsonCompact}
  for \(f_{\alpha}\). Let \(\varepsilon>0\) and \(R>0\) and \(\alpha_0\) as in definition \ref{genHigsonCompact}. Then let \(K\subset \sigma(\M)\) such that \((f_{\alpha_0})|_{K^\complement}\leq\varepsilon\), then \(f_\alpha\leq \chi_K+\varepsilon\) thus \(1-\chi_K \leq 1-f_{\alpha_0}+\varepsilon\). Then, for \(T\) such that \(\|T\|\leq1\) and \(\propagation(T)\leq R\):

  \begin{align}
    \begin{split} \label{inequalityCommutator}
      \|(1-\chi_K)[T,S]\|^2 & =\|[T,S]^\ast(1-\chi_K)^2[T,S]\|                                         \\
                            & \leq \|[T,S]^\ast(1-f_{\alpha_0}+\varepsilon)^2[T,S]\|                   \\
                            & \leq (2\varepsilon+\varepsilon^2)\|[T,S]\|^2+\|(1-f_{\alpha_0})[T,S]\|^2 \\
                            & \leq (2\varepsilon+2\varepsilon^2)\|[T,S]\|^2.
    \end{split}
  \end{align}
  And similarly for \([T,S](1-\chi_K)\).

  Conversely if \(S\) satisfies the condition in the lemma, let \(\varepsilon>0\), \(R>0\) and \(K\) such that \(\|\chi_{K^\complement}[T,S]\|\leq\varepsilon\) for all \(T\) such that \(\propagation(T)\leq R\) and \(\|T\|\leq 1\). Then let \(\alpha_0\) such that \(f_{\alpha}|_{K}\geq 1-\varepsilon\) for \(\alpha\geq\alpha_0\). Then, for \(\alpha\geq \alpha_0\), \(\chi_K\leq f_\alpha+\varepsilon\) and a counterpart to inequality \ref{inequalityCommutator} allows us to conclude.
\end{proof}

\begin{theorem}\label{spechigsoncompact}
  The algebra \(\overline{\A}^H\) is a unital \(C^\ast\)-subalgebra of \(\multiplier(\A)\) that contains \(\A\). In other words it defines a noncommutative coarse structure on \(\A\).
\end{theorem}
\begin{proof}
  The C\(^\ast\) algebra \(\overline{\A}^H\) contains the identity (\([T,Id]=0\)) and is closed under addition (by the triangle inequality), scalar multiplication and the involution (since \(\left( \left( 1-f_\alpha\right)[T,S]\right)^\ast=-[T^\ast,S^\ast](1-f_\alpha)^\ast\) ).

  The algebra \(\overline{\A}^H\) is closed under multiplication. Indeed, let \(\left(f_\alpha\right)\) be an approximate unit of \(\A\) as above. By taking finite convex combinations of elements of this approximate unit, we can assume that it is \emph{quasicentral with respect to \(\multiplier(\A)\)} (see for instance \cite{HigsonRoe} Section 3.2). If \(S_1, S_2\) are slowly oscillating then,
  \begin{align*}
    \left(1-f_\alpha\right)[T,S_1 S_2] & =\left(1-f_\alpha\right)[T,S_1]S_2+\left(1-f_\alpha\right)S_1[T,S_2]                                                                                  \\
                                       & =\underbrace{\left(1-f_\alpha\right)[T,S_1]S_2}_{(1)}+\underbrace{S_1\left(1-f_\alpha\right)[T,S_2]}_{(2)}+\underbrace{[1-f_\alpha,S_1][T,S_2]}_{(3)}
  \end{align*}
  We see that \(S_1S_2\) is slowly oscillating as well since \(S_1\) (resp. \(S_2\)) is slowly oscillating for term (1) (resp. (2)) above, and since \([1-f_\alpha,S_1]\underset{\alpha}{\to} 0\) for term (3), and since a similar reasoning can be done for \([T,S_1S_2](1-f_\alpha)\).

  The algebra \(\overline{\A}^H\) is norm-closed. Indeed, if \(S_n \to S\) in \(\multiplier(\A)\) such that \(S_n\) are slowly oscillating, then for \(R>0\) and \(\varepsilon>0\), let \(n\) such that \(\|S_n-S\|\leq \varepsilon\) and \(K\) compact such that \(\|\chi_{K^\complement}[T,S_n]\|\leq\varepsilon\) and \(\|[T,S_n]\chi_{K^\complement}\|\leq \varepsilon\)  for \(T\) such that \(\propagation(T)\leq R\) and \(\|T\|\leq 1\). Then for such \(T\):
  \begin{align*}
    \|[T,S]\chi_{K^\complement}\|\leq 2\varepsilon
  \end{align*}
  and
  \begin{align*}
    \|\chi_{K^\complement}[T,S]\|\leq 2\varepsilon.
  \end{align*}

  The algebra \(\overline{\A}^H\) contains \(\A\) since for \(S\in\A\),  \((1-f_\alpha)S \to 0\) as \(\alpha\to\infty\), thus \((1-f_\alpha)[T,S] \to 0\).

\end{proof}

\subsection{Relative commutant algebra Higson compactification} 
In this section we consider a different generalization of the Higson compactification, based on the relative commutant Roe algebra, we thus consider \((\A,\M,\lip)\) a proper quantum metric space represented on a Hilbert space \(\HH\). We define the relative commutant Higson compactification as follows:
\begin{definition}[Relative commutant Higson compactification]
  If \((\A,\M,\lip)\) is a proper quantum metric space represented on a Hilbert space \(\HH\), we define the relative commutant Higson compactification of \((\A,\M,\lip)\) to be the set \(\overline{\A}^{H,\comm}\) of elements \(S\) in \(\multiplier(\A)\) such that :
  \[[S,T]\in \K(\HH) \text{ for all } T\in C^{\ast}_{\comm}(\A).\]
\end{definition}

\begin{theorem}\label{commHigsonCompact}
  The algebra \(\overline{\A}^{H,\comm}\) is a unital \(C^\ast\)-subalgebra of \(\multiplier(\A)\) that contains \(\A\). In other words it defines a noncommutative coarse structure on \(\A\).
\end{theorem}
\begin{proof}
  That \(\overline{\A}^{H,\comm}\) contains the identity, is closed for the norm, closed under addition, scalar multiplication and involution is straightforward. That it contains \(\A\) is straightforward from the definition of \(C^{\ast}_{\comm}(\A)\).
\end{proof}

\section{Examples of Coarse quantum metric spaces}
\label{examplesection}
\subsection{The commutative case}
In this section, we verify that in the case of locally compact metric spaces, the spectral Roe algebra corresponds to the usual Roe algebra and the spectral Higson compactification defined above coincides with the usual Higson compactification for proper metric spaces.
\subsubsection{The spectral Roe algebra}
For convenience we recall the definition of propagation in the classical case:
\begin{definition}\label{usualPropagation}
  Let \((X,d)\) be a metric space, and \(\rho:C_0(X)\to \B(\HH)\) be a faithful nondegenerate representation of \(C_0(X)\) on a Hilbert space \(\HH\). For \(T\in \B(\HH)\), its propagation is defined as:
  \[
    \propagation(T)=\sup\{d(x,y):\, (x,y)\in \supp(T)\}
  \]
  where
  \[\supp(T):=X\times X\setminus \bigcup_{\substack{U, V \text{ open }\\ \rho(\chi_U)T\rho(\chi_V)=0}} U \times V.
  \]
\end{definition}
\begin{theorem} \label{spectralRoeAlgebraClassical}
  Let \((X,d)\) be a proper metric space, and consider the Lipschitz triple \((C_0(X),C_0(X),\lip_d)\) where \(\lip_d\) is the usual Lipschitz seminorm. Furthermore, let \(\rho:C_0(X)\to\B(\HH)\) be a faithful nondegenerate representation of \(C_0(X)\).

  Then the propagation \(\propagationS(T)\) defined in section \ref{sectionRoeAlgebras}  coincides with the propagation defined above in definition \ref{usualPropagation}.

  As a consequence, the spectral Roe algebra \(C^{\ast}_{\spec}(C_0(X))\) coincides with the usual Roe algebra \(C^{\ast}(X)\).
\end{theorem}
\begin{proof}
  Let \(T\in \B(\HH)\), we denote by \(\propagationS(T)\) the spectral propagation as defined in section \ref{sectionRoeAlgebras}, and \(\overline{\propagationS}(T)\) the propagation as defined above in definition \ref{usualPropagation} and we show the two inequalities:
  \begin{itemize}
    \item \(\propagationS(T)\leq \overline{\propagationS}(T)\):
          Let \(\overline{\propagationS}(T)\leq R\), in other words for any open \(U,V\) such that \(d(U,V)>R\), we have \(\rho(\chi_U)T\rho(\chi_V)=0\). Let us show that \(T\) has spectral propagation less than or equal to \(R\). If \(f\in \sa(C_0(X))\) is such that \(\lip(f)\leq 1\), then the sets \(U=\{x\in X:\, f(x)\leq \alpha\}\) and \(V=\{x\in X:\, f(x)\geq \beta\}\) are such that \(d(U,V)\geq \beta-\alpha\). Thus if \(\beta-\alpha>R\), then \(\rho(\chi_U)T\rho(\chi_V)=0\), which shows that \(T\) has spectral propagation less than or equal to \(R\).

    \item \(\overline{\propagationS}(T)\leq \propagationS(T)\): Suppose \(T\) has spectral propagation less than or equal to \(R\), we will show that \(\overline{\propagationS}(T)\leq R\). If \(U, V\) are open sets such that \(d(U,V)>R\), then we can find \(f\in \sa(C_0(X))\) such that \(\lip(f)\leq 1\) and \(f|_U=0\) and \(f|_V>R\). It follows that \(\rho(\chi_V)\leq\chi_{(R,+\infty)}(f)\) and \(\chi_{(-\infty,0]}(f)\leq\rho(\chi_U)\). Then by the definition of spectral propagation we have \( \chi_{(-\infty,0]}(f)T\chi_{(R,+\infty)}(f)=0\), thus \(\rho(\chi_U)T\rho(\chi_V)=0\). This shows that \(\overline{\propagationS}(T)\leq R\).
  \end{itemize}
\end{proof}
\subsubsection{The spectral Higson compactification}
Recall first the definition of the Higson compactification for proper metric spaces:
\begin{definition}[Higson compactification for proper metric spaces]\label{HigsonCompact}
  Let \((X,d)\) be a proper metric space.
  \begin{itemize}
    \item For \(f\in C_b(X)\) and \(R>0\) denote \[\Delta_R f(x)=\sup\{|f(x)-f(y)|:\, d(x,y)\leq R\}.\]
    \item A function \(f\in C_b(X)\) is said to be slowly oscillating if for all \(R>0\),
          \[ \Delta_R f(x)\xrightarrow[x\to\infty]{} 0.\]
    \item The Higson compactification of \(X\) is the spectrum of the unital C\(^\ast\)-algebra \(C_h(X)\) generated by the set of slowly oscillating functions.
  \end{itemize}
\end{definition}
\begin{remark}
  Unpacking the definitions, a function is slowly oscillating if and only if for all \(R>0\) and \(\varepsilon>0\), there exists \(K\subset X\) compact such that:

  \begin{align*}
    d(x,y)\leq R \Rightarrow |f(x)-f(y)|\leq \varepsilon \quad \text{for all }  x,y\in X\setminus K
  \end{align*}

\end{remark}
The following theorem motivates our generalization of the Higson compactification to the quantum setting, showing that it is consistent with the definition above.
\begin{theorem}\label{HigsonCompactificationClassical}
  Let \((X,d)\) be a proper metric space, and consider the Lipschitz triple \((C_0(X),C_0(X),\lip_d)\) where \(\lip_d\) is the usual Lipschitz seminorm. Furthermore, let \(\rho:C_0(X)\to\B(\HH)\) be a faithful nondegenerate representation of \(C_0(X)\).

  Then the generalized Higson compactification of \((C_0(X),C_0(X),\lip_d)\) as in definition \ref{genHigsonCompact} coincides with \(\rho(C_h(X))\), where \(C_h(X)\) is the classical Higson compactification as in definition \ref{HigsonCompact}.

  \[\overline{C_0(X)}^{h}=\rho(C_h(X)).\]
\end{theorem}
\begin{proof}
  Before diving in the proofs of each inclusion recall that for any vector \(\xi\in \HH\), the support of \(\xi\) is defined to be the set \(\supp(\xi)\) defined by:
  \[
    \supp(\xi):=X\setminus \bigcup_{\substack{U\text{ open }\\ \rho(\chi_U)\xi=0}} U.
  \]

  Furthermore for \(T\in \B(\HH)\), its support is the set:
  \[\supp(T):=X\times X\setminus \bigcup_{\substack{U\times V \text{ open }\\ \rho(\chi_U)T\rho(\chi_V)=0}} U \times V.
  \]
  And we have that \(\propagation(T)=\sup\{d(x,y): (x,y)\in \supp(T)\}\).

  \textbf{Step 1: From Roe algebra commutants to slow oscillation}

  Suppose that \(S\in \overline{C_0(X)}^{h}\), in particular \(S\in \rho(\M(C_0(X)))=\rho(C_b(X))\), let \(S=\rho(f)\) where \(f\in C_b(X)\) and let us show that \(f\) is slowly oscillating.

  Let \(R>0\) and \(\varepsilon>0\), since  \(\rho(f)\in\overline{C_0(X)}^{h}\), there is \(K\subset X\) compact such that for all \(T\) such that \(\propagation(T)\leq R+1\) and \(\|T\|\leq 1\):
  \begin{align}\label{GenHigsonCompactFormula}\|[T,\rho(f)](1-\chi_K)\|\leq \varepsilon.\end{align}

  We use the following claim:
  \begin{claim*}\label{local_elements}
    For all \(x\in X\) and \(\delta>0\), there exists a unit vector \(\xi_x\in \HH\) such that:
    \[
      \supp(\xi_x)\subset B_\delta(x)
    \]
  \end{claim*}
  \begin{proof}[Proof of claim]
    By nondegeneracy, there is a \(\xi\) such that \(x\in \supp(\xi)\). It then suffices to take:
    \[
      \xi_x=\frac{\rho(\chi_{B_\delta(x)})\xi}{\|\rho(\chi_{B_\delta(x)})\xi\|}
    \]
  \end{proof}
  Now let \(x,y\in X\setminus K\) such that \(d(x,y)\leq R\), we shall prove that \(|f(x)-f(y)|\leq 3\varepsilon\).

  Choose \(0<\delta<1/2\) such that \(B_\delta(x)\subset X\setminus K\) and \(B_\delta(y)\subset X\setminus K\) and:
  \begin{align*}
    |f(x')-f(x)| & \leq \varepsilon \text{ for all } x'\in B_\delta(x), \\
    |f(y')-f(y)| & \leq \varepsilon \text{ for all } y'\in B_\delta(y)
  \end{align*}

  By claim \ref{local_elements}, there exists unit vectors \(\xi_x\) and \(\xi_y\) such that \(\supp(\xi_x)\subset B_\delta(x)\) and \(\supp(\xi_y)\subset B_\delta(y)\). The operator \(T_{xy}=\ket{\xi_x}\bra{\xi_y}\) is supported within \(B_\delta(x)\times B_\delta(y)\) thus \(\propagation(T_{xy})\leq d(x,y)+2\delta\leq R+1\) and \(\|T_{xy}\|=1\).

  Then
  \(\|\left(\rho(f)-f(x)\right)\xi_x\|\leq \varepsilon\) and \(\|\left(\rho(f)-f(y)\right)\xi_y\|\leq \varepsilon\).

  Thus by equation \ref{GenHigsonCompactFormula}:
  \begin{align*}
    \varepsilon \geq|\braket{\xi_y|[T_{xy},\rho(f)](1-\chi_K)\xi_x}| & = |\braket{\xi_y|[T_{xy},\rho(f)]|\xi_x}|                                           \\&=|\braket{\xi_y|T_{xy}|\rho(f)\xi_x}-\braket{\rho(\bar{f})\xi_y|T_{xy}|\xi_x}|\\
                                                                     & \geq |f(x)\braket{\xi_y|T_{xy}|\xi_x}-f(y)\braket{\xi_y|T_{xy}|\xi_x}|-2\varepsilon \\
                                                                     & =|f(x)-f(y)|-2\varepsilon
  \end{align*}
  Thus \(|f(x)-f(y)|\leq 3\varepsilon\) as needed.

  \textbf{Step 2: From slow oscillation to Roe algebra commutants}

  If \(f\) is slowly oscillating, let us show that it satisfies the condition \ref{genHigsonCompact}.
  We have the following fact (recorded for instance in \cite{willettHigherIndexTheory2020}) , that can be proven essentially in the same way as in proposition \ref{fromSpectralToRelativeCommutant}
  \begin{claim*}[Commutation inequality]
    For any \(f\in C_{b}(X)\), \(T\in\B(\HH)\):
    \[
      \|[\rho(f),T]\|\leq 8\omega_{\propagationS(T)}(f)\|T\|.
    \]
    Where for \(R\geq0\), \(\omega_{R}(f)\) is defined by:
    \[
      \omega_{R}(f)=\sup\{|f(x)-f(y)|\, x,y\in X, d(x,y)<R \}
    \]
  \end{claim*}
  Now for \(\varepsilon>0\) and \(R>0\), let \(K\) such that:
  \begin{align*}
    d(x,y)\leq R \Rightarrow |f(x)-f(y)|\leq \varepsilon \quad \text{for all }  x,y\in X\setminus K
  \end{align*}
  Thus if \(T\in C^{\ast}(X)\) such that \(\propagationS(T)\leq R\) and \(\|T\|\leq 1\):
  consider the corners of \(\rho\) and \(T\) by the projection \(\rho(\chi_{K^c})\) and the restriction \(f_{K^c}\) of \(f\) to \(K^c\), we then have \(\omega_{R}(f|_{K^c})\leq \varepsilon\), thus applying the commutation inequality:

  \[ \|\rho(\chi_{K^c})\left[\rho(f|_{K^c}),\rho(\chi_{K^c})T\rho(\chi_{K^c})\right]\rho(\chi_{K^c})\|\leq \varepsilon.\]
  Thus
  \[ \|\rho(\chi_{K^c})\left[\rho(f),T\right]\rho(\chi_{K^c})\|\leq \varepsilon.\]
  Now since \([T,\rho(f)]\) has propagation less than or equal \(R\), for a compact \(L\) such that \(\overline{B_R(K)}\subset L\):

  \begin{align*} \|\left[\rho(f),T\right]\rho(\chi_{L^c})\| & =\|\rho(\chi_{B_R(L^c)})[\rho(f),T]\rho(\chi_{L^c})\|                                  \\
                                                          & \leq\|\rho(\chi_{K^c})\left[\rho(f|_{K^c}),T\right]\rho(\chi_{K^c})\|\leq \varepsilon.
  \end{align*}
\end{proof}
\subsection{Compact Quantum Metric spaces}
For compact quantum metric spaces represented on a Hilbert space \(\HH\), as is expected, all operators in \(\B(\HH)\) have finite spectral propagation, finite commutant seminorm, and finite topographic propagation.

The coarse structure is of course trivial since the C\(^\ast\)-algebra is unital thus \(\A=\overline{\A}^{h}=\multiplier(\A)\).

First recall from \cite{latremoliereQuantumLocallyCompact} that a unital lipschitz pair \((\A,\lip)\) is a compact quantum metric space if and only if \((\A,\lip,\mathbb{C}\mathbbm{1})\) is a locally compact quantum metric space. We thus can consider a compact quantum metric space in the sense of Rieffel, as a locally compact quantum  metric space.

\begin{prop}
  Let \((\A,\lip)\) be a compact quantum metric space and \(\rho:\A\to \B(\HH)\) be a representation of \(\A\), and let \(R_0\) be its diameter, then for any operator \(T\) in \(\B(\HH)\)
  \begin{itemize}
    \item The spectral propagation of \(T\) is bounded by \(R_0\): \(\propagation(T)\leq R_0\).
    \item The commutation seminorm of \(T\) is bounded.
  \end{itemize}
\end{prop}
\begin{proof}
  If \(T\in\B(\HH)\) and \(a\in \A\) is self-adjoint such that \(\lip(a)\leq 1\), let us show that  \(\chi_{(-\infty,0]}(a)T\chi_{(R_0,+\infty)}(a)=0.\)

  \begin{claim*}
    For \(\xi\in \HH\), either \(\chi_{[R_0,+\infty)}(a)\xi=0\) or \(\chi_{(-\infty,0)}(a)\xi=0\).
  \end{claim*}

  \begin{proof}[Proof of claim]
    Note that by the definition of the diameter as the diameter of the space of states \(S(\A)\) under the Monge-Kantorovich distance,
    \[
      |\mu(a)-\nu(a)|\leq R_0 \text{  for all } \mu \text{ and } \nu \text{ states on }
      \A.
    \]
    For \(\xi\in \HH\) denote \(\xi_+=\chi_{[R_0,+\infty)}(a)\xi\) and \(\xi_- = \chi_{(-\infty,0)}(a)\xi\), and suppose, by absurd that \(\xi_+\neq 0\) and \(\xi_-\neq 0\), it follows then by considering the states: \(a\mapsto \frac{\braket{\xi_+,\rho(a)\xi_+}}{\|\xi_+\|^2}\) and \(a\mapsto \frac{\braket{\xi_-,\rho(a)\xi_-}}{\|\xi_-\|^2}\) that
    \[
      \frac{\braket{\xi_+,\rho(a)\xi_+}}{\|\xi_+\|^2}- \frac{\braket{\xi_-,\rho(a)\xi_-}}{\|\xi_-\|^2}\leq R_0
    \]
    Furthermore since by the definition of \(\xi_+\) and \(\xi_-\) we have that
    \begin{align*}
      \frac{\braket{\xi_+,\rho(a)\xi_+}}{\|\xi_+\|^2}\geq R_0,
      \frac{\braket{\xi_-,\rho(a)\xi_-}}{\|\xi_-\|^2}< 0,
    \end{align*}

    thus
    \begin{align*}\frac{\braket{\xi_+,\rho(a)\xi_+}}{\|\xi_+\|^2}- \frac{\braket{\xi_-,\rho(a)\xi_-}}{\|\xi_-\|^2}>R_0,
    \end{align*}

    a contradiction, and it follows that: \(\xi_+=0\) or \(\xi_-=0\).
  \end{proof}

  Now if \(T\in \B(\HH)\) and \(\xi\in \HH\):
  \begin{align*}
    \braket{\xi|\chi_{(-\infty,0)}(a)T\chi_{[R_0,+\infty)}(a)\xi} & =\braket{\chi_{(-\infty,0)}(a)\xi|T\chi_{[R_0,+\infty)}(a)\xi}=0,
  \end{align*}
  since by the claim above either \(\chi_{(-\infty,0)}(a)\xi=0\) or \(\chi_{[R_0,+\infty)}(a)\xi=0\). It follows that \(\chi_{(-\infty,0)}(a)T\chi_{[R_0,+\infty)}(a)=0\).
\end{proof}

\subsection{Coarse disjoint unions of Compact Quantum Metric spaces}
\subsubsection{General setting}
We consider the coarse disjoint union of a sequence of compact quantum metric spaces \((\A_n,\lip_n)\) where \(\A_n\) is unital for all \(n\). We endow the direct sum \(\bigoplus_{n}\A_n\) with a Lip-norm and a topography that make it into a locally compact quantum metric space, one of the motivations is the study of quantum expanders.
\begin{definition}[Coarse disjoint union of compact quantum metric spaces]
  Let \((\A_n,\lip_n)\) be a sequence of compact quantum metric spaces. Fix \(\varphi_n\) states on \(\A_n\), and \(R_n\) a sequence of positive real numbers that satisfies \(R_n\to\infty\).

  We define the coarse disjoint union of \((\A_n,\lip_n)\), with distances \(R_n\) between states \(\varphi_n\), to be the Lipschitz triple \((\A,\M,\lip)\) where:
  \begin{itemize}
    \item \(\A=\bigoplus_{n}\A_n\),
    \item \(\M=\bigoplus_{n}\C\unit_n\) is the topography,
    \item \(\lip\) is defined for  \(a=(a_n)_n\in \sa(\A)\) by:
          \begin{align*}
            \lip((a_n)_n)=\max\left(\sup_n \lip_n(a_n),\sup_n \frac{|\varphi_{n+1}(a_{n+1})-\varphi_n(a_n)|}{R_n}\right)
          \end{align*}
  \end{itemize}
\end{definition}
In order to show that this is indeed a locally compact quantum metric space, the main technical hurdle is to show that the Monge-Kantorovich metric metrizes the weak* topology on the set of (tame) states.  Alternatively it suffices to show the characterization stated in theorem \ref{locallycompactqmscharacterization} in terms of precompactness of certain subsets of \(\uu\A\).

We will ultimately reduce to using this characterization for the compact parts \(\A_n\), in the following form:
\begin{lemma}\label{strongercharacterizationcompact}
  Let \((\A_0,\lip_0)\) be a compact quantum metric space, \(\varphi_0\) a state on \(\A_0\) and \(I\) a bounded subset of \(\R\). Then the set:
  \begin{align*}
    \{a\in \sa(\A):\, \lip_0(a)\leq 1, \varphi_0(a)\in I\}
  \end{align*}
  is norm precompact in \(\A_0\).
\end{lemma}
\begin{proof}
  Denote \(E_I=\{a\in \sa(\A):\, \lip_0(a)\leq 1, \varphi_0(a)\in I\}\) for any \(I\subset \R\).

  Now assume \(I\) is bounded. By the fundamental characterization of compact quantum metric spaces \cite{Rieffel1999},  \(E_0=    \{a\in \sa(\A):\, \lip_0(a)\leq 1, \varphi_0(a)=0\}\) is norm precompact in \(\A\). The function \(E_0\times \R \to E_{\R}\) defined by \((a,t)\mapsto a+t\unit\) is continuous and maps the compact set \(\overline{E_0}\times \overline{I}\) onto a compact set containing \(E_I\), thus \(E_I\) is precompact in \(\A\).
\end{proof}
\begin{theorem}\label{coarsedisjointunion}
  Any coarse disjoint union of compact quantum metric spaces \((\A_n,\lip_n)\) is a locally compact quantum metric space.

  Furthermore if \((\A_n,\lip_n)\) are separable, \(\lip\)  is lower semi-continuous defined on a dense subset of \(\A_n\) (as compared with \(\sa(\A_n)\)) ,and satisfy the Leibniz inequality, then their coarse disjoint is strongly proper, in the sense of \cite{latremoliereTopographicGromovHausdorffQuantum2014}, Definition 3.2.12.
\end{theorem}
\begin{proof}
  \textbf{Step 1: weak*-metrizability condition}
  Denote by \((\A,\M,\lip)\) the coarse disjoint union of \((\A_n,\lip_n)\) as above, let \(\varphi_n\) and \(R_n\) be respectively the states and distances used in the definition of \(\lip\).

  Recall that by theorem \ref{locallycompactqmscharacterization} to show that the Lipschitz triple \((\A,\M,\lip)\) is a locally compact quantum metric space, it suffices to show that, for a local state \(\psi\in \mathcal{S}(\A|\M)\) for all \(s,t\in \M\) with compact support, the set:
  \begin{align*}
    \{sat:\,a\in\sa(\uu\A) , \lip(a)\leq 1, \psi(a)=0\}
  \end{align*}
  is norm precompact in \(\A\).

  Since \(\M=\bigoplus_n \C\unit_n\cong c_0(\N)\), compactly supported elements of \(\M\) are finite sequences corresponding to elements of finite sums \(\bigoplus_{k=0}^n \C \unit_k\), and correspondingly, local states are states supported on finite sums \(\bigoplus_{k=0}^n \A_k\).

  In particular \((a_n)_n\mapsto \varphi_0(a_0)\) defines a local state on \(\A\).

  Thus in the setting above, it suffices to show that  for any \(n\in \N\), the set:
  \begin{align*}
    E_0=\{(a_0,\dots,a_n,0,\dots)\in \sa(\uu\A):\, \lip((a_0,\dots,a_n,a_{n+1},\dots))\leq 1, \varphi_0(a_0)=0\}
  \end{align*}
  is norm precompact in \(\A\).

  The condition \[\lip((a_0,\dots,a_n,\dots))\leq 1, \varphi_0(a_0)=0\] implies \[L_k(a_k)\leq 1, |\varphi_{k+1}(a_{k+1})-\varphi_k(a_k)|\leq R_k \text{ for } 0\leq k\leq n-1 \text{, and } \varphi_0(a_0)=0\] which in turn implies \[L_k(a_k)\leq 1 \text{ and } |\varphi_{k}(a_k)|\leq \sum_{0\leq k'\leq k-1} R_k \text{ for } 0\leq k\leq n-1.\]
  Thus \(E_0\) can be continuously embedded in:
  \begin{align*}
    \prod_{0\leq k \leq n }\{a_k\in \sa(\uu\A_k):\, \lip_k(a_k)\leq 1, |\varphi_{k}(a_{k})|\leq C\}
  \end{align*}
  which are precompact sets by Lemma \ref{strongercharacterizationcompact}. Thus \(E_0\) is precompact.

  \textbf{Step 2: Strong properness}:
  By definition \ref{stronglyproperqms}, it suffices to show that there exists an approximate unit \((e_n)_n\) in  \(\sa(\M)\) for \(\A\), comprized of elements with compact supports in \(\M\) and such that \(\lip(e_n )\xrightarrow[n\to\infty]{} 0\). Choosing \(e_n=(\unit_0,\dots,\unit_n,0,\cdots)\), this is compactly supported in the topography and \(\lip(e_n)=\frac{1}{R_n}\xrightarrow[n\to\infty]{} 0 \).

\end{proof}


\subsection{Almost-commutative Lipschitz triples}
Riemannian manifolds and in particular spin manifolds provide a rich source of metric spaces. In noncommutative geometry, the counterpart, and one of the motivating examples for C*-quantum metric spaces, are spectral triples. In this context, some of the simplest examples after commutative spectral triples (i.e. spin manifolds), which led to models for the Standard Model of particle physics, are almost-commutative spectral triples, which are tensor products of commutative spectral triples with finite-dimensional spectral triples. That is, triples of the form \((C^{\infty}(M)\hat{\otimes}\A_F,L^2(M,S)\hat{\otimes}\HH_F,\slashed{D}\hat{\otimes}\unit+\epsilon\hat{\otimes}D_F )\) where
\begin{itemize}
  \item \(M\) is a compact spin manifold, \(S\) the spinor bundle over \(M\), \(\slashed{D}\) the associated Dirac operator, and \(\epsilon\) the (constant) graduation of \(L^2(M,S)\) in the even case (even spinors are \(+1\) eigenvectors, and odd spinors are \(-1\)), and \(\epsilon=1\) in the odd case,
  \item \(\A_F\) is a finite dimensional C*-algebra,
  \item \(\HH_F\) is a finite dimensional Hilbert space,
  \item \(D_F\) is a self-adjoint operator on \(\HH_F\), making \((\A_F,\HH_F,D_F)\) a finite dimensional spectral triple.
\end{itemize}
In this model the algebra \(\A_F\) encodes the internal degrees of freedom of the elementary particles, while the manifold \(M\) encodes the spacetime. Since we are interested in the metric coarse structure we shall consider models inspired by these construction where:
\begin{itemize}
  \item \(M\) is a complete Riemannian manifold (not necessarily compact).
  \item We focus on the quantum metric structure (and forget about other structures carried by the spinor bundle and the Dirac operator).
\end{itemize}
In this context, to motivate the right class of Lipschitz seminorms to consider on \(C_0(M)\hat{\otimes}\A_F\), we note that all Lipschitz seminorms on finite dimensional C*-algebras are equivalent (in the sense that they are comparable up to multiplicative constants), since the formula \(\inf\{\|a-c\unit\mid \text{ for } c\in \mathbb{C}\|\}\) defines a Lipschitz seminorm on \(\A_F\), we obtain the following:
\begin{prop}
  For any Lipschitz seminorm \(\lip_F\) on a finite dimensional C*-algebra \(\A_F\), there exists \(C>0\) such that for all \(a\in \sa(\A_F)\):
  \begin{align*}
    \frac{1}{C}\inf_{c\in \mathbb{C}}\|a-c\unit\|\leq \lip_F(a)\leq C\inf_{c\in \mathbb{C}}\|a-c\unit\|
  \end{align*}
\end{prop}
\begin{proof}
  Since \(\A_F\) is finite dimensional, all norms on \(\sa(\A_F)/\mathbb{C}\unit\) are equivalent, in particular the quotient norm induced by the C*-norm and the norm induced by \(\lip_F\).
\end{proof}
Based on this observation we fix the following notation:
\begin{definition}
  Let \(\A_F\) be a finite dimensional C\(^\ast\)-algebra, we denote by \(\lip_{0}\) the Lipschitz seminorm defined for \(a\in \sa(\A_F)\) by:
  \begin{align*}
    \lip_{\A_F}(a)=\inf_{c\in \mathbb{C}}\|a-c\unit\|
  \end{align*}
\end{definition}
Furthermore, almost-commutative spectral triples where shown to satisfy a triangle inequality (even in the non unital case) \cite{dandreaPythagorasTheoremProducts2013}, of the form:
\begin{align*}
  d(\varphi\otimes \varphi_F,\psi\otimes \psi_F)\leq d_M(\varphi,\psi)+d_F(\varphi_F,\psi_F)
\end{align*}
For states \(\varphi,\psi\) on \(C_0(M)\) and \(\varphi_F,\psi_F\) on \(\A_F\), where \(d_M\) and \(d_F\) are the Monge-Kantorovich metrics on \(M\) and on \(\A_F\) respectively. In fact, in the unital (compact) case, a lower bound is also satisfied in the form of a Pythagoras inequality. Here we shall only consider the upper bound.

Based on the observations above we define the following class of almost-commutative Lipschitz triples:
\begin{definition}[Almost-commutative Lipschitz triple]
  Let \(X\) be a proper metric space and let \(\A_F\) be finite dimensional C\(^\ast\)-algebra.

  An almost-commutative Lipschitz triple on \(X\) with fibers modeled on \(\A_F\) is:
  \begin{itemize}
    \item A Lipschitz triple \((C_0(X)\otimes\A_F,C_0(X)\otimes\mathbb{C}\unit,\lip)\) such that:
    \item    There exists \(C>0\), such that for \(F\in C_0(X)\otimes\A_F\), whenever \(\lip(F)<\infty\):
          \begin{align*}
            \lip(F)\leq C\left(\sup_{x\in X}\left(\lip_0(F(x))\right)+\lip_X(F)\right)
          \end{align*}
    \item The Monge-Kantorovich metric \(d\) associated to \((C_0(X)\otimes\A_F,C_0(X)\otimes\mathbb{C}\unit,\lip)\) satisfies the following triangle inequality: for any states \(\varphi,\psi\) on \(C_0(X)\) and \(\varphi_F,\psi_F\) on \(\A_F\):
          \begin{align*}
            d(\varphi\otimes \varphi_F,\psi\otimes \psi_F)\leq C \left(d_X(\varphi,\psi)+d_F(\varphi_F,\psi_F)\right)
          \end{align*}
          where \(d_X\) and \(d_F\) are the Monge-Kantorovich metrics associated to \((C_0(X),C_0(X),\lip_d)\) and \((\A_F,\mathbb{C}\unit,\lip_0)\) respectively.
  \end{itemize}
\end{definition}

\begin{remark}
  We make the following straightforward (but nonetheless important) observations about the definition above:
  \begin{itemize}
    \item In the above definition, recalling that \(C_0(X)\otimes\A_F\cong C_0(X,\A_F)\), \(\lip_X\) denotes:
          \[\lip_X(F)=\sup_{x\neq y}\left(\frac{\|F(x)-F(y)\|}{d(x,y)}\right)\]
    \item It follows automatically that \(\lip\) is densely defined.
    \item Since \(\A_F\) is finite dimensional, states on \(C_0(X)\otimes\A_F\) can be written as finite sums of tensor products of states on \(C_0(X)\) and states on \(\A_F\). This, combined with the fact that the topography is \(C_0(X)\otimes\mathbb{C}\unit\) allows us to show that tame sets of states on \(C_0(X)\otimes\A_F\) are constructed from tame sets of states on \(C_0(X)\), thus the triangle inequality on product states suffices to show that the Monge-Kantorovich metric metrizes the weak* topology on tame sets of states on \(C_0(X)\otimes\A_F\): any almost-commutative Lipschitz triple is a locally compact quantum metric space.
    \item Given the above we shall also call \((C_0(X)\otimes\A_F,C_0(X)\otimes\mathbb{C}\unit,\lip)\) a locally compact almost-commutative quantum metric space.
    \item If \(X\) is proper and \(\lip\) is lower semi-continuous, satisfying Leibniz inequality, then \((C_0(X)\otimes\A_F,C_0(X)\otimes\mathbb{C}\unit,\lip)\) is strongly proper in the sense of \cite{latremoliereTopographicGromovHausdorffQuantum2014}.
  \end{itemize}
\end{remark}

\begin{prop}
  If \((C_0(M)\otimes \A_F,L^2(M,S),D)\) is an almost-commutative spectral triple on a (nonnecessarily compact) spin manifold then, with \(\lip\) the Lipschitz seminorm defined by \(\|[D,\cdot]\|\), \((C_0(M)\otimes\A_F,C_0(M)\otimes\mathbb{C}\unit,\lip)\) is an almost-commutative Lipschitz triple.
\end{prop}


\subsubsection{Spectral propagation and Roe algebras:}
Let us first fix a representation of \(C_0(X)\otimes\A_F\) on a Hilbert space. Let \(\rho_X\) be an ample representation of \(C_0(X)\) on \(\HH_X\) and \(\rho_F\) be a faithful representation of \(\A_F\) on \(\HH_F\), and consider:
\begin{itemize}
  \item The representation \(\rho=\rho_X\otimes\rho_F\) of \(C_0(X)\otimes\A_F\) on \(\HH_X\otimes\HH_F.\)
  \item For any operator \(T\in \B(\HH_X\otimes\HH_F)\), and any \(\xi_1,\xi_2\in \HH_F\), we denote by \(T_{\xi_1 \xi_2}\in \B(\HH_X)\) the operators defined by:
        \begin{align*}
          \braket{\eta_1|T_{\xi_1 \xi_2}|\eta_2}=\braket{\eta_1\otimes \xi_1|T|\eta_2\otimes \xi_2} \text{ for all } \eta_1,\eta_2\in \HH_X.
        \end{align*}
  \item In particular if \(e_i\) is an orthonormal basis of \(\HH_F\), we can write \(T\) as a matrix of operators \((T_{ij})_{i,j}\) with entries in \(\B(\HH_X)\), where \(T_{ij}=T_{e_i e_j}\).
\end{itemize}
We want to relate the spectral propagation of \(T\) with respect to \(C_0(X)\otimes\A_F\), to the spectral propagation of its blocks \(T_{ij}\) with respect to \(C_0(X)\).

\begin{prop}\label{propagationalmostcommutative}
  Let \((C_0(X)\otimes\A_F,C_0(X)\otimes\mathbb{C}\unit,\lip)\) be an almost-commutative locally compact metric space represented on  \(\HH_X\otimes\HH_F\). Let \(R_0\) be the radius of the quantum metric space \((\A_F,\lip_F)\).
  Then there exists \(C>0\) such that for \(T\in \B(\HH_X\otimes\HH_F)\):
  \begin{align*}
    \frac{1}{C}\sup_{i,j}\left[\propagationS(T_{ij})\right] \leq \propagationS(T) \leq C \left(\sup_{i,j}\left[\propagationS(T_{ij})\right]+R_0\right)
  \end{align*}
\end{prop}
Let us first recall the following notation, which we shall use in the second part of the proof, for \(\eta\in\HH_X\), let \(\supp(\eta)\) be the closed subset of \(X\) defined by:
\begin{align*}
  \supp(\xi) = \{x \in X : \text{for all open } U \ni x, \text{ we have } \rho(\chi_U)\xi \neq 0\}
\end{align*}
\begin{proof}
  By the definition of almost-commutative Lipschitz triples, there exists \(C>0\) such that for all \(f\in C_0(X)\otimes\A_F\) with \(\lip(f)\leq 1\):
  \begin{align*}
    \frac{1}{C}\left(\sup_{x\in X}\left(\lip_0(f(x))\right)+\lip_X(f)\right)\leq \lip(f)\leq C\left(\sup_{x\in X}\left(\lip_0(f(x))\right)+\lip_X(f)\right)
  \end{align*}
  and for any states \(\varphi,\psi\) on \(C_0(X)\) and \(\varphi_F,\psi_F\) on \(\A_F\):
  \begin{align*}
    d(\varphi\otimes \varphi_F,\psi\otimes \psi_F)\leq C \left(d_X(\varphi,\psi)+d_F(\varphi_F,\psi_F)\right)
  \end{align*}
  \begin{itemize}
    \item \(\sup_{i,j}\left[\propagationS(T_{ij})\right]\leq C\propagationS(T)\): \\ Let \(R=\propagationS(T)\),  if \(f\in C_0(X)\) is such that \(\lip_X(f)\leq 1\), then \(\lip_X(f\otimes 1)\leq 1\) thus \(\lip(f\otimes \unit)\leq C\), it follows that:
          \[\chi_{(-\infty,0)}(f\otimes \unit)T\chi_{[CR,+\infty)}(f\otimes \unit)=0\]
          Which implies that \(\chi_{(-\infty,0)}(f) T_{ij}\chi_{[CR,+\infty)}(f)=0\) for all \(i,j\), thus \(\propagationS(T_{ij})\leq CR\) for all \(i,j\).
    \item \(\propagationS(T) \leq C \left(\sup_{i,j}\left[\propagationS(T_{ij})\right]+R_0\right)\):\\ Let \(R=\sup_{i,j}\left[\propagationS(T_{ij})\right]\), it follows more generally that for all \(\xi_1,\xi_2\in \HH_F\), \(\propagationS(T_{\xi_1 \xi_2})\leq R\). Now if \(F\in C_0(X)\otimes\A_F\) is such that \(\lip(F)\leq 1\), then if \(x,y\in X\) and \(\xi_1,\xi_2\in \HH_F\) unit vectors:
          \begin{align*}
            |\braket{\xi_1,F(x)\xi_1}-\braket{\xi_2,F(y)\xi_2}|\leq C(d(x,y)+R_0)
          \end{align*}
          by the triangle inequality applied to the states \(F\mapsto \braket{\xi_1,F(x)\xi_1}\) and \(F\mapsto \braket{\xi_2,F(y)\xi_2}\) on \(C_0(X)\otimes\A_F\).
          Thus if \(\eta_1, \eta_2\in \HH_X\) and \(\xi_1,\xi_2\in \HH_F\) are such that
          \begin{align*}
            \eta_1\otimes \xi_1\in  \left(\unit\otimes\ket{\xi_1}\bra{\xi_1}\right)\chi_{(-\infty,0)}(F)(\HH_X\otimes\HH_F) , \\ \eta_2\otimes \xi_2\in  \left(\unit\otimes\ket{\xi_2}\bra{\xi_2} \right)\chi_{[C(R+R_0),+\infty)}(F)(\HH_X\otimes\HH_F)
          \end{align*}
          Then
          \begin{align*}
            \braket{\eta_1\otimes\xi_1,F\eta_1\otimes\xi_1}\leq0, \\
            \braket{\eta_2\otimes\xi_2,F\eta_2\otimes\xi_2} > C(R+R_0).
          \end{align*}
          Which shows that \(C(d(\supp(\eta_1),\supp(\eta_2))+R_0)\geq C(R+R_0)\), thus:
          \begin{align*}
            d(\supp(\eta_1),\supp(\eta_2))\geq R,
          \end{align*}
          and it follows that:
          \begin{align*}
            \braket{ \eta_1\otimes\xi_1| T|\eta_2\otimes \xi_2} & =\braket{\xi_1|T_{\xi_1,\xi_2}|\xi_2}=0
          \end{align*}
          By the propagation condition on \(T_{\xi_1 \xi_2}\). Thus \(\chi_{(-\infty,0)}(F)T\chi_{[C(R+R_0),+\infty)}(F)=0\), and it follows that \(\propagationS(T)\leq C(R+R_0)\).
  \end{itemize}
\end{proof}
The following is then straightforward:
\begin{theorem}
  For an almost-commutative Lipschitz triple \((C_0(X)\otimes\A_F,C_0(X)\otimes\mathbb{C}\unit,\lip)\) represented on \(\HH_X\otimes\HH_F\) :
  \begin{align*}
    C^{\ast}_{\spec}(C_0(X)\otimes\A_F)\cong C^{\ast}(X)\otimes \B(\HH_F)
  \end{align*}
\end{theorem}
\subsubsection{Spectral Higson compactification and coarse structure:}
In this section we describe the spectral Higson compactification of almost-commutative Lipschitz triples. We first note that since \(\A_F\) is unital, \(\multiplier(C_0(X)\otimes \A_F)\cong C_b(X)\otimes \A_F \cong C_0(\beta X)\otimes \A_F\). The following computation then follows straightforwardly from the definitions:
\begin{prop}
  Let \((C_0(X)\otimes\A_F,C_0(X)\otimes\mathbb{C}\unit,\lip)\) be an almost-commutative Lipschitz triple, then with respect to a representation of the form \(\HH_X\otimes\HH_F\), the spectral Higson compactification, as a subset of \(C_0(X)\otimes \A_F\), is given by:
  \begin{align*}
    \overline{C_0(X)\otimes\A_F}^h=C_0(X)\otimes \A_F + C_h(X)\otimes \mathbb{C}\unit
  \end{align*}
  Where \(C_0(X)\otimes\A_F+ C_h(X)\otimes \mathbb{C}\unit\) is identified with its image under the representation on \(\HH_X\otimes\HH_F\).
\end{prop}
\begin{proof}
  We show the double inclusion:
  \begin{itemize}
    \item \(\overline{C_0(X)\otimes\A_F}^h\subset C_0(X)\otimes \A_F + C_h(X)\otimes \mathbb{C}\unit\): Fix an orthonormal basis \((e_i)_i\) for \(\HH_F\). Let \(S\in \overline{C_0(X)\otimes\A_F}^h\), and denote by \(S_{ij}\in \B(\HH_X)\) its blocks with respect to the basis \((e_i)_i\). By applying the definition of the spectral Higson compactification,  since operators of the form \(1\otimes \ket{e_i}\bra{e_j}\) have finite spectral propagation by proposition \ref{propagationalmostcommutative}
          \[[1\otimes \ket{e_i}\bra{e_j},S]\left(\rho_X(\chi_{K^{\complement}})\otimes\unit\right)\xrightarrow[K\to X]{}0\]
          \(\left(1\otimes \ket{e_i}\bra{e_j}\right)S\) written as a matrix of operators on \(\HH_X\) has entries \(S_{j,l}\) in the \(i\)-th row and \(l\)-th column, and \(0\) elsewhere, while \(S\left(1\otimes \ket{e_i}\bra{e_j}\right)\) has entries \(S_{k,i}\) in the \(k\)-th row and \(j\)-th column, and \(0\) elsewhere.  It follows by taking \(i\neq j\) that the \((i,i)\) entry of \([\left(1\otimes \ket{e_i}\bra{e_j}\right),S]\) is \(S_{j,i}\) and thus:
          \begin{align*}
            S_{j,i}\rho_X(\chi_{K^{\complement}})\xrightarrow[K\to X]{}0 \text{ for } j\neq i.
          \end{align*}
          Thus \begin{align}\label{approximatelyfinitedimensionalHigson1}S_{j,i}\in C_0(X) \text{ for } j\neq i.\end{align}
          Furthermore the \((i,j)\) entry of \([\left(1\otimes \ket{e_i}\bra{e_j}\right),S]\) is \(S_{j,j}-S_{i,i}\), thus similarly:
          \begin{align}\label{approximatelyfinitedimensionalHigson2}S_{j,j}-S_{i,i}\in C_0(X) \text{ for } j\neq i.\end{align}
          Combining \eqref{approximatelyfinitedimensionalHigson1} and \eqref{approximatelyfinitedimensionalHigson2}, we obtain that:
          \begin{align*}
            S-S_{11}\otimes \unit\in C_0(X)\otimes \A_F.
          \end{align*}
          Furthermore, if \(T\in \B(\HH_{X})\) then \(\propagationS(T\otimes \unit)\leq C\left(\propagationS(T)+R_0\right)\) for some \(C>0\) and \(R_0\) independent of \(T\) by proposition \ref{propagationalmostcommutative}, thus since \(S\in \overline{C_0(X)\otimes \A_F}^h\), we obtain that for any \(R>0\) and \(\varepsilon\) there exists \(K\) such that:
          \[\|[S,T\otimes \unit]\chi_{K^\complement}\|\leq \varepsilon, \|\chi_{K^\complement}[S,T\otimes \unit]\|\leq \varepsilon \quad \text{ for any } T \in (\B(\HH_X))_1 \text{ with } \propagationS(T)\leq R \] thus taking the \((1,1)\) element, with the identification in the commutative case (theorem \ref{HigsonCompactificationClassical}) \[S_{11}\in C_h(X) \] and it follows that \(S=\left(S-S_{11}\otimes \unit\right)+S_{11}\otimes \unit\in C_0(X)\otimes \A_F + C_h(X)\otimes \mathbb{C}\unit\).
    \item \(C_0(X)\otimes \A_F + C_h(X)\otimes \mathbb{C}\unit\subset \overline{C_0(X)\otimes\A_F}^h\):\\
          It suffices to show that for \(f\in C_h(X)\), \(f\otimes \unit\in \overline{C_0(X)\otimes\A_F}^h\). Let \(R>0\) and \(\varepsilon>0\), by the characterization of Higson functions from theorem \ref{HigsonCompactificationClassical}, there exists a compact set \(K\subset X\) such that for all \(\tilde{T}\in \left(\B(\HH_X)\right)_1\) with \(\propagationS(\tilde{T})\leq CR\):
          \[\|[f,\tilde{T}]\chi_{K^\complement}\|\leq \varepsilon/\dim(\HH_F), \quad \|\chi_{K^\complement}[f,\tilde{T}]\|\leq \varepsilon/\dim(\HH_F).\]
          It follows that for all \(T\in \B(\HH_X\otimes \HH_F)\) with \(\propagationS(T)\leq R\), \(\propagationS\left(T_{ij}\right)\leq CR\) (proposition \ref{propagationalmostcommutative}), thus since:
          \[[f\otimes \unit,T]\chi_{K^\complement}=([f,T_{ij}]\chi_{K^\complement})_{ij}, \quad \chi_{K^\complement}[f\otimes \unit,T]=(\chi_{K^\complement}[f,T_{ij}])_{ij}\]
          we have :
          \[\|[f\otimes \unit,T]\chi_{K^\complement}\|\leq \varepsilon, \quad \|\chi_{K^\complement}[f\otimes \unit,T]\|\leq \varepsilon.\]
  \end{itemize}
\end{proof}

Looking at the Higson coronas: \(\overline{C_0(X)\otimes \A_F}^{h}/C_0(X)\otimes \A_F\cong C_{h}(X)/C_0(X)=C(\nu X)\), we naturally expect that the spectral coarse structure on \(C_0(X)\otimes \A_F\) is coarsely isomorphic to \(C_0(X)\). 

A notion of coarse equivalence, and more generally coarse map has been introduced in \cite{banerjeeNoncommutativeCoarse}, let us recall the definitions first. For \((\A,\overline{\A})\) a noncommutative coarse space, we denote by \(\partial \A=\overline{\A}/\A\) its corona.
\begin{definition}[Noncommutative coarse map]
  If \((\A,\overline{\A})\) and \((\BB,\overline{\BB})\) are noncommutative coarse spaces, a noncommutative coarse map from \((\A,\overline{\A})\) to \((\BB,\overline{\BB})\) is a unital completely positive map \(\overline{\varphi}:\overline{\A}\to \overline{\BB}\) continuous with respect to the strict topology of \(\overline{\A}\) and \(\overline{\BB}\) such that:
  \begin{itemize}
    \item \(\varphi=\overline{\varphi}_{\mid \A}:\A\to \BB\) maps \(\A\) to \(\BB\),
    \item the induced map on the coronas \(\partial \varphi:\partial \A\to \partial \BB\) is a *-homomorphism.
  \end{itemize}\end{definition}
\begin{definition}[Noncommutative coarse equivalence]
  For two noncommutative coarse spaces \((\A,\overline{\A})\) and \((\BB,\overline{\BB})\) and a noncommutative coarse map \(\overline{\varphi}:\overline{\A}\to \overline{\BB}\), \(\overline{\varphi}\) is a coarse equivalent if there exists a noncommutative coarse map \(\overline{\psi}:\overline{\BB}\to \overline{\A}\) such that the compositions \(\overline{\psi}\circ \overline{\varphi}\) and \(\overline{\varphi}\circ \overline{\psi}\) are close to the respective identity maps, i.e. the induced maps on the coronas are the identity maps.
\end{definition}
\begin{theorem}\label{almostcommutativecoarseequivalence}
  Let \((C_0(X)\otimes\A_F,C_0(X)\otimes\mathbb{C}\unit,\lip)\) be an almost-commutative Lipschitz triple, then with respect to a representation of the form \(\HH_X\otimes\HH_F\), the spectral coarse structure on \(C_0(X)\otimes \A_F\) is coarse equivalent to the coarse structure on \(C_0(X)\).
  \[C_0(X)\otimes\A_F \underset{\text{coarse}}{\cong}C_0(X)\]
\end{theorem}
\begin{proof}
  Let \(\tau:\A_F\to\mathbb{C}\) a (tracial) state on \(\A_F\).
  Define \(\overline{\varphi}:\overline{C_0(X)\otimes\A_F}^h\to C_h(X)\) and \(\overline{\psi}:C_h(X)\to \overline{C_0(X)\otimes\A_F}^h\) by:
  \begin{align*}
    \overline{\varphi}(f\otimes a)=\tau(a)f, \quad \overline{\psi}(f)=f\otimes \unit.
  \end{align*}
  Then:
  \begin{align*}
    \overline{\varphi}\circ \overline{\psi}(g)              & =g,                                                           \\
    \overline{\psi}\circ \overline{\varphi}(g\otimes \unit) & = g\otimes \unit,                                             \\
    \overline{\psi}\circ \overline{\varphi}(f\otimes a)     & =\tau(a)f \otimes \unit = f\otimes a \mod C_0(X)\otimes \A_F, \\
  \end{align*}
  for  \(f\in C_0(X)\), \(g\in C_h(X)\), \(a\in \A_F\). Recalling that \(\overline{C_0(X)\otimes\A_F}^h \cong C_0(X)\otimes \A_F + C_h(X)\otimes \mathbb{C}\unit \), the induced maps on the coronas are the identity maps.
\end{proof}
\begin{remark}
  The commutative version of the above theorem is the fact that a space \(Y\) made of finitely many copies of a proper metric space \(X\), whith uniformly bounded distances between the copies of \(X\), satisfies \(Y \underset{\text{coarse}}{\cong} X\).
\end{remark}

\subsection{Dirac operators in spectral triples and finite Lipschitz commutant seminorm}
In this section we consider general (nonunital) spectral triples \((\A,\HH,D)\), computing the Roe algebra (spectral or relative commutant) is expected to be more subtle in general (than in the previous section) but we can show that if we fix the Lipschitz seminorm defined by \(\lip(a)=\|[D,a]\|\), then if one takes a bounded function \(\varphi\) whose Fourier transform is compactly supported, then \(\varphi(D)\) has finite Lipschitz commutant seminorm with respect to \((\A, \lip)\). In particular this shows that \(\varphi(D)\in C^{\ast}_{\comm}(\A)\) for any \(\varphi\in C_0(\R)\), thus, at least we know that \(C^{\ast}_{\comm}(\A)\) is not trivial. Furthermore these results will be used to construct the higher index of \(D\) in section \ref{higherindexsection}.

To keep the exposition as general as possible, we will make minimal assumptions on the spectral triple \((\A,\HH,D)\), in particular we make no mention of summability or regularity conditions, nor of real structures. As we are mainly concerned with the relative commutant Lipschitz seminorm, we will not fix any topography on \(\A\) either.

We thus make the following assumptions:
\begin{itemize}
  \item \(\A\) is a C\(^\ast\)-algebra represented faithfully on the Hilbert space \(\HH\),
  \item \(D\) is an unbounded self-adjoint operator on \(\HH\) such that the set \[\{a\in \A : [D,a] \text{ is densely defined and extends to a bounded operator on } \HH\}\] contains a dense *-subalgebra \(\mathcal{A}\) of \(\A\).
  \item \(D\) satisfies local compactness: \(a(1+D^2)^{-1/2}\) is compact for all \(a\in \A\).
\end{itemize}
We denote by \(\lip(a)\) the Lipschitz seminorm defined for \(a\in \sa(\mathcal{A})\) by:
\begin{align*}
  \lip(a)=\|[D,a]\|
\end{align*}
We then have the following:

\begin{prop}\label{spectraltriplefinitecommutanteitD}
  In the above setting, for any \(t\in \R\), the operator \(e^{itD}\) has finite Lipschitz commutant seminorm with respect to \((\A,\lip)\):
  \begin{align*}
    \lip^\ast(e^{itD})\leq |t|
  \end{align*}
\end{prop}
\begin{proof}
  For \(a\in \sa(\mathcal{A})\), we have:
  \begin{align*}
    \frac{d}{dt}\left(e^{itD}a e^{-itD}\right) & = i e^{itD}[D,a]e^{-itD}.
  \end{align*}
  It follows that:
  \begin{align*}
    \|[e^{itD},a]\| & =\|e^{itD}a e^{-itD}-a\|=\left\|\int_0^t \frac{d}{ds}\left(e^{isD}a e^{-isD}\right) ds\right\| \\
                    & \leq |t|\|[D,a]\|= |t|\lip(a).
  \end{align*}
  Since \(\mathcal{A}\) is dense in \(\A\) for the C\(^\ast\)-norm, it follows that \(\lip^\ast(e^{itD})\leq |t|\).
\end{proof}
It follows that:
\begin{prop}\label{spectraltriplefinitecommutantechiD}
  In the above setting, for any \(\varphi:\R\to \mathbb{C}\) a bounded Borel function whose Fourier transform \(\hat{\varphi}\) is compactly supported, the operator \(\varphi(D)\) has finite Lipschitz commutant seminorm with respect to \((\A,\lip)\), and more precisely:
  \begin{align*}
    \lip^\ast(\varphi(D))\leq \int_{\R} |t||\hat{\varphi}(t)| dt
  \end{align*}
\end{prop}
\begin{proof}
  Indeed by the Fourier inversion formula:
  \begin{align*}
    \varphi(D)=\int_{\R} \hat{\varphi}(t)e^{itD} dt
  \end{align*}
  where the integral converges in the strong operator topology. Thus for \(a\in \sa(\mathcal{A})\):
  \begin{align*}
    [\varphi(D),a] & =\int_{\R} \hat{\varphi}(t)[e^{itD},a] dt
  \end{align*}
  where the integral converges in the strong operator topology, and it follows that:
  \begin{align*}
    \|[\varphi(D),a]\| & \leq \int_{\R} |\hat{\varphi}(t)|\|[e^{itD},a]\| dt \leq \left(\int_{\R} |t||\hat{\varphi}(t)| dt\right)\lip(a).
  \end{align*}
  which shows that:
  \begin{align*}
    \lip^\ast(\varphi(D))\leq \int_{\R} |t||\hat{\varphi}(t)| dt.
  \end{align*}
\end{proof}
Now since \(C_0(\R)\) is generated by \(\frac{1}{\sqrt{1+x^2}}\) we see that by local compactness of \(D\), \(\varphi(D)\in \B(\HH)\) is locally compact with respect to \(\A\) for any \(\varphi\in C_0(\R)\): \(a\varphi(D)\) and \(\varphi(D)a\) are compact for all \(a\in \A\).

Furthermore since the set of functions in \(C_0(\R)\) whose Fourier transform is compactly supported is dense in \(C_0(\R)\), it follows that \(\varphi(D)\in C^{\ast}_{\comm}(\A)\) for any \(\varphi\in C_0(\R)\).
\begin{crl}\label{locallycompactfinitecommutant}
  In the above setting, for any \(\varphi\in C_0(\R)\), we have:
  \begin{align*}
    \varphi(D)\in C^{\ast}_{\comm}(\A).
  \end{align*}
\end{crl}

\section{Higher index theory}
\label{higherindexsection}
In this section, we fix a proper locally compact quantum metric space \((\A,\M,\lip)\) represented non-degenerately and amply on a Hilbert space \(\HH\). Using our definition of Roe algebras, we deduce some aspects of Higher index theory for locally compact quantum metric spaces.

As operators of finite spectral propagation with respect to the full algebra are difficult to construct in general we shall consider the following two types of Higher index:
\begin{itemize}
  \item The first version is valued in the relative commutant Roe algebras \(C^{\ast}_{\comm}(\A)\) defined in section \ref{sectionRoeAlgebras}, for which one can use the functoriality properties. (Note that the larger \(C^{\ast}_{q}(\A)\) in the commutative case  has been used to define the Higher index of some pseudodifferential operators). While this Higher index can be defined in general for Dirac operators coming from spectral triples, in order to define it for general elements of the K-homology of \(\A\), we shall need to assume the existence of a partition of unity that commutes sufficiently well with elements of \(\A\).   
  \item The second version is valued in the spectral Roe algebra with respect to an arbitrary commutative subalgebra \(\M\subset \BB\subset \A\), (for instance \(\BB=\M\)) for which one can construct finite propagation operators by the usual cutting procedure.
\end{itemize}
We first introduce the dual algebras \(D_{\comm}^{\ast}(\A)\) (respectively  \(D_{\spec,\BB}^{\ast}(\A)\)) and show that under some assumptions on the regularity of commutants in \(\A\) with respect to the Lipschitz norm, the K-homology of \(\A\) can be obtained from the K-theory of \(D_{\comm}^{\ast}(\A)/C_{\comm}^{\ast}(\A)\) (respectively \(D_{\spec,\BB}^{\ast}(\A)/C_{\spec,\BB}^{\ast}(\A)\)), using the 6-term exact sequence of K-theory thus allows us to define a Higher index \(\mathrm{Ind}:K^{i}\left(\A \right)\to K_{i}\left(C^{\ast}_{\comm}(\A)\right)\) (respectively \(C^{\ast}_{\spec,\BB}(\A)\)).
\subsection{Relative commutant Roe algebras valued index}
\begin{definition}
  Denote by \(D^{\ast}_{\comm}(\A)\) the \(C^\ast\)-algebra generated by operators on \(\HH\) that commute with \(\A\) up to compacts and are of finite relative commutant seminorm.
  \begin{align*}
    D^\ast_{\comm}(\A)=\overline{\{T\in \B(\HH):\,\lip^\ast(T) < \infty \textrm{ and } \forall a\in \A: [T,a]\in\K(\HH)\}}
  \end{align*}
\end{definition}

We further recall the following \(C^{\ast}\)-algebras used in Paschke duality:
\begin{itemize}
  \item \(\mathcal{C}(\A)=\{T \in \B(\HH):\, \forall a\in \A: Ta \in \K(\HH), aT\in \K(\HH) \}\) the set of operators that are locally compact.
  \item \(\mathcal{D}(\A)=\{T \in \B(\HH):\, \forall a\in \A: [T,a] \in \K(\HH)\}\) the set of operators that commute up to compacts with \(\A\).
\end{itemize}
Then it is known that:
\begin{theorem}[Paschke duality]
  \[ K_{i+1}\left(\mathcal{D}(\A)/\mathcal{C}(\A)\right)=K^i(\A).\]
\end{theorem}

\begin{definition}
  We shall say that \((\A,\M,\lip)\) has an approximate unit satisfying the \textbf{commutant-Lipschitz} condition if there exists a sequence of strictly positive constants \(C_n\) such that \(C_n\underset{n \to \infty}{\to} 0\) and an approximate unit  \(e_n\) made of local elements in \(\M\) such that:
  \begin{equation*}
    \|[e_n,a]\|\leq C_n \lip(a) \text{ for all } a \in \A
  \end{equation*}
  and for any compact set \(K\) in \(\sigma(\M)\) there exists \(N\) such that \(e_N\chi_K=\chi_K\).
\end{definition}
Note that the above definition entails that for any \(n\), there exists \(K_n \subset L_n\) compact sets in \(\sigma(\M)\) such that \(e_n\chi_{K_n}=\chi_{K_n}\) and \(e_n\chi_{L_n^\complement}=0\).

Furthermore the above holds if \(\A\) satisfies the inequality \(\|[a,b]\|\leq C\lip(a)\lip(b)\) for some \(C\) and there exists an approximate unit \(e_n\) made of local elements in \(\M\) such that \(\lip(e_n)\to 0\) and, for any compact set \(K\) in \(\sigma(\M)\), \(e_N\chi_K=1\) for some \(N\).

Let us first record the following consequence of the characterization of locally compact quantum metric spaces and of the existence of commutant-Lipschitz approximate units:

\begin{lemma}
  If \((\A,\M,\lip)\) a locally compact quantum metric space, \(\mu\in \state(\A|\M)\) and \((e_n)_n\) an approximate unit of \(\A\) in \(\M\) satisfying the commutant-Lipschitz condition, then for all \(n\):
  \begin{equation*}
    \{ e_n a  \mid a\in \sa(\uu\A), \lip(a) \leq 1, \mu(a)=0 \}
  \end{equation*}
  is totally bounded for the norm topology of \(\A\).
\end{lemma}
\begin{proof}
  Let \(\varepsilon>0\) we shall show that     \(\{ e_n a  \mid a\in \sa(\uu\A), \lip(a) \leq 1, \mu(a)=0 \}\) can be covered by a finite number of balls of radius less than or equal \(\varepsilon\). Let \(K\) be a compact set such that \(e_n\chi_{K^\complement}=0\) (i.e. \(e_n\) is supported within \(K\)). Let then \(N\) such that \(e_N\chi_{K}=\chi_{K}\). (i.e. \(e_N=1\) on \(K\)), and such that \(\|[a,e_N]\|\leq \varepsilon \lip(a)\). Such \(K\) and \(N\) exist by assumption. By the characterization of locally compact quantum metric spaces (theorem \ref{locallycompactqmscharacterization}), the set
  \begin{equation*}
    \{e_n a e_N \mid a\in \sa(\uu\A), \lip(a) \leq 1, \mu(a)=0 \}
  \end{equation*}
  is totally bounded for the norm topology of \(\A\).
  It follows that there exists a finite set \(F\subset \sa(\uu\A)\) with \(\lip(b)\leq 1\) and \(\mu(b)=0\) for all \(b\in F\) such that for all \(a\in \sa(\uu\A)\) with \(\lip(a)\leq 1\) and \(\mu(a)=0\), there exists \(b\in F\) such that:
  \[
    \|e_n a e_N - e_n b e_N\| \leq \varepsilon.
  \]
  We then have, since \(e_n e_N= e_n\):
  \[
    \|e_n a  - e_n b \|\leq \|e_n [e_N,a] \| + \|e_n[e_N, b] \| + \|e_n a e_N - e_n b e_N\| \leq  3\varepsilon.
  \]
  Which shows the desired total boundedness.
\end{proof}

If \(\A\) has an approximate unit satisfying the commutant-Lipschitz condition then in the above equality we can replace \(\mathcal{D}(\A)\) by \(D^{\ast}_{\comm}(\A)\) and  \(\mathcal{C}(\A)\) by \(C^{\ast}_{\comm}(\A)\):
\begin{theorem}
  \label{cutcommutationroe}
  If \((\A,\M,\lip)\) has an approximate unit satisfying the commutant-Lipschitz condition then the following equality of \(C^{\ast}\) algebras holds:
  \[
    D^{\ast}_{\comm}(\A)/C^{\ast}_{\comm}(\A)=\mathcal{D}(\A)/\mathcal{C}(\A).
  \]
\end{theorem}
Before diving in the proof we recall the following lemma from \cite{localization_algebras} that will be useful in the proof:
\begin{lemma}\label{commutationlemma}
  Let \(I\) be an ideal of a \(C^*\)-algebra \(L\). For any separable \(C^*\)-algebra \(D \subset \mathcal{I}L\) there is a
  positive contraction \(x \in \mathcal{I}I\) such that:
  \begin{itemize}
    \item \([x, d] \in \mathcal{I}_0I\) for all \(d \in D\),
    \item \((1 - x)d \in \mathcal{I}_0I\) for all \(d \in D \cap \mathcal{I}I\).
  \end{itemize}
\end{lemma}

\begin{proof}[Proof of theorem \ref{cutcommutationroe}]
  It suffices to show that \(\mathcal{D}(\A)\subset D^{\ast}_{\comm}(\A)+\mathcal{C}(\A)\). Let \(e_n\) be an approximate unit of \(\A\) in \(\M\) with    \[\|[e_n,a]\|\leq C_n \lip(a) \text{ for all } a \in \A.\] By eventually taking a subsequence of \((e_n)_n\), we can assume \(C_n\leq \frac{1}{2^{n+2}} \). Let \(f_n=e_{n}-e_{n-1}\) with \(f_1=e_1\). 
  We therefore have that:  \[\|[f_n,a]\|\leq \frac{\lip(a)}{2^n} \text{ for all } a \in \A, n \in \N \text{, and } \sum_n f_n=1 \text{ strictly}.\]

  If \(T\in \mathcal{D}(\A)\), then applying lemma \ref{commutationlemma} for the separable \(C^{\ast}\)-algebra \(D\) generated by \(T\) and \(\A\), viewed as constant elements of \(\mathcal{I}\B(\HH)\), and with \(I=\K(\HH)\) viewed as an ideal of \(L=\B(\HH)\).  We get that there exists \(x\in \mathcal{I}\K(\HH)\) such that:
  \begin{itemize}
    \item \([x ,d]\in \mathcal{I}_0 \K(\HH)\) for all \(d\in D\),
    \item \((1-x) d \in \mathcal{I}_0 \K(\HH) \) for all \(d \in D \cap \mathcal{I}\K(\HH)\).
  \end{itemize}
  taking \(y=1-x\), and denote by \(y_t\) the evaluation at \(t\), we have in particular that:
  \begin{enumerate}
    \item \label{cutycommutation} \([y_t T,a]\underset{t\to\infty}{\to} 0\) for all \(a\in \A\):\\ Indeed \([y_t T,a]=(1-x_t)[T,a]+[x_t,a]T\), looking at the terms of the sum, \([T,a]\) is an element of  \(D\cap \mathcal{I}\K(\HH)\) thus \((1-x_t)[T,a]\to 0\) and \(a\) is in \(D\) thus \([x_t,a]T\to 0\).
    \item \label{cutycompact} \((1-y_t)T\) is compact (thus locally compact):\\
          Indeed \((1-y_t)T=x_t T\) compact since \(x\in \mathcal{I}\K(\HH)\).
  \end{enumerate}

  Now we are going to construct an operator \(\tilde{T}\) of the form \[\tilde{T}=\sum_{n}  y_{t_n} T f_n,\] for a well chosen sequence \(t_n\to\infty\), such that: \[\lip^{\ast}(\tilde{T})<\infty \text{  and  }T-\tilde{T}\in \mathcal{C}(\A).\]

  Fix \(\mu\) a local state on \(\A\). For any \(n\in \mathbb{N}\), by the lemma above the set
  \[
    \{ f_n a  : a\in \sa(\uu\A), \lip(a)\leq 1, \mu(a)=0\}
  \]
  is totally bounded since it is a sum of two totally bounded sets. Now, since \([y_t T, f_n a ]\underset{t}{\rightarrow} 0\) for all \(a\in \sa(\uu\A)\) (property (\ref{cutycommutation}) above), we can choose \(t_n\) large enough such that for all \(a\in \sa(\uu\A)\) with \(\lip(a)\leq 1\) and \(\mu(a)=0\):
  \[\|[y_{t_n} T, f_n a] \|\leq \frac{1}{2^n}.\]
  By eventually taking a larger \(t_n\), we can furthermore require that:
  \[\|[y_{t_n}T,f_n]\|\leq \frac{1}{2^n}.\]
  Then we have:
  \[
    a y_{t_n} T f_n = f_n  a y_{t_n} T  + [a,f_n] y_{t_n} T + a [y_{t_n} T,f_n].
  \]

  It follows that, for \(a\in \sa(\uu\A)\) with \(\lip(a)\leq 1\) and \(\mu(a)=0\):
  \begin{align*}
    \|[ y_{t_n} T f_n,a]\| & =\| y_{t_n} T f_n a - a y_{t_n} T f_n  \|                                                   \\
                           & \leq \|(y_{t_n} T f_n a-f_n  a y_{t_n} T) \|+ \|[a,f_n] y_{t_n} T \|+ \|a [y_{t_n} T,f_n]\| \\
                           & \leq \|[y_{t_n} T,f_n a]\|+\|[a,f_n] y_{t_n} T \|+ \|a [y_{t_n} T,f_n]\|                    \\
                           & \leq \frac{3}{2^n}
  \end{align*}
  By summing over \(n\) it follows that
  \[[\tilde{T},a]\leq 3 \text{ for all } a \in \sa(\uu\A) \text{ with } \lip(a)\leq 1 \text{ and } \mu(a)=0.\]
  Now if \(a\in \sa(\A)\) with \(\lip(a)\leq 1\), then setting \(\tilde{a}=a-\mu(a)1\), we have that \(\lip(\tilde{a})=\lip(a)\leq 1\) and \(\mu(\tilde{a})=0\), thus:
  \[
    \|[\tilde{T},a]\|=\|[\tilde{T},\tilde{a}+\mu(a)1]\|=\|[\tilde{T},\tilde{a}]\|\leq 3.
  \]
  It follows that:
  \(\lip^\ast(\tilde{T})\leq 3<\infty\).

  Finally, we have that:
  \[
    T-\tilde{T}=\sum_n T f_n- y_{t_n} T f_n=\sum_n(1-y_{t_n})T f_n.
  \]
  If \(a\) is an element of \(\chi_K\A\chi_K\cap \A\) for a compact \(K\) in \(\M\), then by the assumption on the approximate unit \(e_n\), there exists \(N\) such that for all \(n\geq N\), \(f_n a=0\). It follows that
  the sum \(\sum_n (1-y_{t_n})T f_n a\) is a finite sums of compact terms, (see property (\ref{cutycompact}) above), from which we see \(T-\tilde{T}\in \mathcal{C}(\A).\)
  We conclude that \(\mathcal{D}(\A)\subset D^{\ast}_{\comm}(\A)+\mathcal{C}(\A)\) and thus the two quotients are equal.
\end{proof}


\begin{crl}
  With \(\A\) as in theorem \ref{cutcommutationroe}, we have the following isomorphism:
  \[
    K_{i+1}\left(D^{\ast}_{\comm}(\A)/C^{\ast}_{\comm}(\A)\right)=K^i\left(\A\right)
  \]
\end{crl}
Now, if we consider the following short exact sequence,
\[
  0 \rightarrow C^{\ast}_{\comm}(\A) \rightarrow D^{\ast}_{\comm}(\A)\rightarrow D^{\ast}_{\comm}(\A)/C^{\ast}_{\comm}(\A)  \rightarrow 0
\]
we get a boundary map:
\[
  \mathrm{Ind}:K_{i+1}(D^{\ast}_{\comm}(\A)/C^{\ast}_{\comm}(\A)  )\to K_i(C^{\ast}_{\comm}(\A)).
\]

In particular if \(\A\) satisfies the conditions of theorem \ref{cutcommutationroe}, we get a map:
\[
  \mathrm{Ind}:K^{i}(\A)\to K_i(C^{\ast}_{\comm}(\A)).
\]
\begin{definition}
  We call the above map the Higher index map of \(\A\) valued in the relative commutant Roe algebra of \(\A\).
\end{definition}

In particular if \((\A,\HH,D)\) is a nonunital spectral triple (and without additional assumption on \(\A\) or on the eventual topography) we can construct an element of \(D^{\ast}_{\comm}(\A)\), and thus a Higher index associated with dirac operator \(D\) as follows. Let us first recall the following denomniation:
\begin{definition}
  A bounded function \(\psi:\R\to\R\) is called a normalizing function if it is odd, and satisfies \(\lim_{x\to \infty} \psi(x)=1\) and \(\lim_{x\to -\infty} \psi(x)=-1\).
\end{definition}

\begin{prop}\label{higherindexspectraltriple}
  Let \((\A,\HH,D)\) be a nonunital spectral triple, and let \(\psi\) be a normalizing function then \(\psi(D)\in D^{\ast}_{\comm}(\A)\), furthermore the class \([\psi(D)]\in D^{\ast}_{\comm}(\A)/C^{\ast}_{\comm}(\A)\) does not depend on the choice of \(\psi\).

  It follows that we get an element \(\mathrm{Ind}(D)=\mathrm{Ind}(\psi(D))\in K_i(C^{\ast}_{\comm}(\A))\).
\end{prop}

\begin{proof}

  Recall first that if \(\psi_1\) and \(\psi_2\) are two normalizing functions then \(\psi_1-\psi_2 \in C_0(\R)\), thus \(\psi_1(D)-\psi_2(D)\in C^{\ast}_{\comm}(\A)\).

  \textbf{The operator \(\psi(D)\) commutes up to compacts with \(\A\):}
  We reproduce for completeness what we need from the proof in \cite{careyIndexTheoryLocally2014} (section 3) and \cite{kaadKKTheorySpectralFlow2012} (theorem 4.1).
  Consider the normalizing function \(\psi_0: x\mapsto \frac{x}{\sqrt{1+x^2}}\)  then for \(a\) and \(b\) such that \([D,a]\) is bounded we have:
  \[
    [\psi_0(D),a]b= [D(1+D^2)^{-\frac{1}{2}},a]b = [D,a](1+D^2)^{-\frac{1}{2}}b + D[(1+D^2)^{-\frac{1}{2}},a]b
  \]
  \([D,a](1+D^2)^{-\frac{1}{2}}b\) is compact since it is the product of a bounded operator with a compact one. Furthermore using the integral formula for the functional calculus of \((1+D^2)^{-\frac{1}{2}}\) we have:
  \begin{align*}
    D[(1+D^2)^{-\frac{1}{2}},a]b = \frac{1}{\pi}\int_0^{\infty} \lambda^{-\frac{1}{2}} & D^2(\lambda + 1 + D^2)^{-1}[D,a](\lambda + 1 + D^2)^{-1}b d\lambda + \\ &\frac{1}{\pi}\int_0^{\infty} \lambda^{-\frac{1}{2}} D(\lambda + 1 + D^2)^{-1}[D,a](\lambda + 1 + D^2)^{-1}D b d\lambda
  \end{align*}
  Each of the two integrals above converges in operator norm and is compact since it is the norm limit of compact operators (using that \((\lambda + 1 + D^2)^{-1}b\) is compact). Thus \([\psi_0(D),a]b\) is compact, and by density of \(\mathcal{A}\) in \(\A\), we see that it is compact for any \(a\) and \(b\) in \(\A\).  It follows \([\psi_0(D),ab]=[\psi_0(D),a]b+a[\psi_0(D),b]\) is compact as well. By density of products in \(\A\), it follows that \([\psi_0(D),a]\) is compact for any \(a\in \A\).

  For a general normalizing function \(\psi\), we can write \(\psi=\psi_0 + \phi\) where \(\phi\in C_0(\R)\) thus \(\phi(D)\) is locally compact, it follows that \([\psi(D),a]\) is compact for any \(a\in \A\).

  \textbf{The operator \(\psi(D)\) is in \(D^{\ast}_{\comm}(\A)\):}
  Taking a normalizing function \(\psi_1\) whose Fourier transform is compactly supported (such a function exists), we have by proposition \ref{spectraltriplefinitecommutantechiD} that \(\lip^\ast(\psi_1(D))<\infty\). It follows by the previous step that \(\psi_1(D)\in D^{\ast}_{\comm}(\A)\).

  Now, again for a general normalizing function \(\psi\), we can write \(\psi=\psi_1 + \phi\) where \(\phi\in C_0(\R)\) thus \(\phi(D)\in C^\ast_{\comm}(\A)\) by corollary \ref{locallycompactfinitecommutant}, it follows that \(\psi(D)\in D^{\ast}_{\comm}(\A)\).
\end{proof}

\subsection{Lipschitz controlled K-theory and quantitative K-theory, a comparison for finite Assouad-Nagata dimension}
In this brief section, we note that the above construction of localized relative commutant Roe algebras, along with theorem \ref{assouadnagatarelativecommutantroetheorem}, allows to compare two different approaches to approximate K-homology of metric spaces, namely quantitative K-theory \cite{quantitativeKtheory} \cite{toyotaControlledKtheoryKhomology2025} and Lipschitz controlled K-theory \cite{lipschitzKtheory}. With the filtration given by the propagation \(\propagationS\), quantitative K-theory of the Roe algebra was used to approximate K-homology. While Lipschitz controlled K-theory \cite{lipschitzKtheory} was mainly used to approximate K-theory of metric spaces, but if we instead consider the Lipschitz controlled K-theory for the relative commutant algebras introduced above, with Lipschitz seminorm \(L^\ast\), and in the case of metric spaces of finite Assouad-Nagata dimension, then this is equivalent to using quantitative K-theory to approximate the K-homology of the space as in \cite{toyotaControlledKtheoryKhomology2025} for instance.

We invite the interested reader to refer to \cite{lipschitzKtheory} and \cite{quantitativeKtheory} \cite{toyotaControlledKtheoryKhomology2025} for the definitions of these two notions of controlled K-theory, the result then follows immediately from theorem \ref{assouadnagatarelativecommutantroetheorem}.

\subsection{Index valued in the Roe algebra of a commutative subalgebra}
Let \(\BB\) be commutative subalgebra of \(\A\), such that \(\M\subset \BB\subset \A\). Then we have defined in section \ref{sectionRoeAlgebras} the propagation with respect to \(\BB\) and the corresponding Roe algebra \(C^{\ast}_{\spec,\BB}(\A)\cong C^{\ast}_{\spec}(\BB)\). We can then define the following dual algebra:
\begin{definition}
  Denote by \(D^{\ast}_{\spec,\BB}(\A)\) the \(C^\ast\)-algebra generated by operators on \(\HH\) that commute with \(\A\) up to compacts and are of finite spectral propagation with respect to \(\BB\).
  \begin{align*}
    D^{\ast}_{\spec,\BB}(\A)=\overline{\{T\in \B(\HH):\,\propagationS_{\BB}(T) < \infty \textrm{ and } \forall a\in \A: [T,a]\in\K(\HH)\}}
  \end{align*}
\end{definition}


Again we have the following:
\begin{theorem}
  \label{cuttorelativeroe}
  The following equality of \(C^{\ast}\) algebras holds:
  \[
    D^{\ast}_{\spec,\BB}(\A)/C^{\ast}_{\spec,\BB}(\A)=\mathcal{D}(\A)/\mathcal{C}(\A).
  \]
\end{theorem}
\begin{proof}
  It suffices to show that \(\mathcal{D}(\A)=D^{\ast}_{\spec,\BB}(\A)+\mathcal{C}(\A)\). In other words, it suffices to show that any operator that commutes with \(\A\) up to compact can be decomposed into a "finite propagation" part and a locally compact part.

  Since \(\A\) is proper, \(\sigma(\BB)\) is a proper separable metric space (in the classical sense), we can fix a partition of unity \(\chi_n\in \BB\) such that \({\rm diam}\left({\rm supp}\left(\chi_n \right)\right)\leq 1\).

  Now for \(T\in \mathcal{D}(\A)\), let
  \[T'=\sum_n \sqrt{\chi_n}T\sqrt{\chi_n}. \]
  Then
  \begin{itemize}
    \item \(T'\) is well-defined and \(\propagation(T')\leq 1\):

          Indeed for \(a\in \mathfrak{Loc}(\A|K)\) where \(K\) is compact in \(\sigma(\M)\), then the sum \(\sqrt{\chi_n}T\sqrt{\chi_n}a\) has a finite number of non-zero terms, and we can conclude strong convergence of the sum since \(\mathfrak{Loc}(\A|\ast)\) is dense and the representation is non-degenerate.

          Furthermore if \(f, g\in \BB\) have supports of distance at least 1, then \(f\sqrt{\chi_n}T\sqrt{\chi_n}g=0\), since \(f\) having support in \(\supp(\chi_n)\) would imply that \(g\) does not and vice versa. Thus \(fT'g=0\) and \(\propagation(T')\leq 1\).

    \item \(T-T'\) is locally compact. Indeed
          \[(T-T')a=\sum_n ( T\chi_n -\sqrt{\chi_n}T\sqrt{\chi_n}a)=\sum_n [T,\sqrt{\chi_n}]\sqrt{\chi_n}a
          \]
          For \(a\in \mathfrak{Loc}(\A|\ast)\) this is a finite sum of compact operators and is thus a compact. We can conlude since \(\mathfrak{Loc}(\A|\ast)\) is norm-dense and \(\K(\HH)\) is closed.
  \end{itemize}
  Note that in the above we get that furthermore \(T'\) commmutes up to compacts with \(\A\), since \(T\) and \(T-T'\) both do. Thus \(T'\in D^{\ast}_{\spec,\BB}(\A)\) and \(T\in \mathcal{C}(\A).\)

  Furthermore, it is clear each \(\sqrt{\chi_n}T\sqrt{\chi_n}\) has finite spectral propagation with respect to \(\BB\), thus so does \(T'\).
\end{proof}

\begin{remark}
  The only difference with the case where \(\A\) is commutative is that we need to assume that \(\mathfrak{Loc}(\A|\ast)\) is dense, this is true for proper quantum metric spaces.
\end{remark}

\begin{crl}
  \[
    K_{i+1}\left(D^{\ast}_{\spec,\BB}(\A)/C^{\ast}_{\spec,\BB}(\A) \right)=K^i\left(\A\right)
  \]
\end{crl}
Considering the following short exact sequence,
\[
  0 \rightarrow C^{\ast}_{\spec,\BB}(\A) \rightarrow D^{\ast}_{\spec,\BB}(\A)\rightarrow D^{\ast}_{\spec,\BB}(\A)/C^{\ast}_{\spec,\BB}(\A)  \rightarrow 0
\]
We get boundary maps:
\[
  \mathrm{Ind}:K^i(\A)\to K_i(C^{\ast}_{\spec,\BB}(\A))=K_i\left(C^{\ast}\left(\sigma(\BB)\right)\right).
\]

\subsubsection{K-homology through finite propagation valued localization algebras}
\label{sectionlocalizedrelativecommutantroetheory}
\begin{definition}
  Define \(C^{\ast}_{L,\spec,\BB}(\A)\) to be the closure of the set of uniformly continuous bounded maps \(f:[0,\infty)\to C^{\ast}_{\spec,\BB}(\A)\), with bounded propagation with respect to \(\BB\), such that for all \(a\in \A\), \([f(t),a]\underset{t\to\infty}{\rightarrow}  0\).
\end{definition}

We then have the following:
\begin{theorem}\label{relative_localized_roe_algebras_thm}
  We have a natural isomorphism:
  \[K_i\left(C^{\ast}_{L,\spec,\BB}(\A)\right)\cong K^{i}(\A) \].

  Under the above isomorphisms the coarse index is induced by the evaluation map \(\mathrm{ev}_0:C^{\ast}_{L,\spec,\BB}(\A)\to C^{\ast}_{\spec,\BB}(\A)\cong C^{\ast}(\sigma(\BB))\) given by \(\mathrm{ev}_0(f)=f(0).\)
  \[
    \mathrm{Ind}=(\mathrm{ev}_0)_{\ast}:K_i\left(C^{\ast}_{L,\spec,\BB}(\A)\right)\to K_i\left(C^{\ast}(\sigma(\BB))\right)
  \]
\end{theorem}

We first introduce the corresponding ``compactly commuting'' algebras:

\begin{definition}
  Define \(D^{\ast}_{L,\spec,\BB}(\A)\) to be the closure of the set of uniformly continuous bounded maps \(f:[0,\infty)\to D^{\ast}_{\spec,\BB}(\A)\) such that for all \(a\in \A\), \([f(t),a]\underset{t\to\infty}{\rightarrow}  0\).
\end{definition}

Then we prove the following path regularization lemma:
\begin{lemma}\label{path_regularization}
  \[D^{\ast}_{L,\spec,\BB}(\A)/C^{\ast}_{L,\spec,\BB}(\A)=\mathcal{I}D^{\ast}_{\spec,\BB}(\A)/\mathcal{I}C^{\ast}_{\spec,\BB}(\A)\]
\end{lemma}
\begin{proof}
  It suffices to show that \(\mathcal{I}D^{\ast}_{\spec,\BB}(\A)=D^{\ast}_{L,\spec,\BB}(\A)+\mathcal{I}C^{\ast}_{\spec,\BB}(\A)\).
  If \(f\) is in \(\mathcal{I}D^{\ast}_{\spec,\BB}(\A)\) we construct \(\tilde{f}\in D^{\ast}_{L,\spec,\BB}(\A)\) such that \(f-\tilde{f}\) is in \(\mathcal{I}C^{\ast}_{\spec,\BB}(\A)\).
  For this, let \(B^{\ast}\) be the \(C^{\ast}\)-algebra generated by operators of finite propagation with respect to \(\BB\), and let \(K^{\ast}=B^{\ast}\cap \K(\HH)\). The algebra \(K^{\ast}\) is an ideal of \(B^{\ast}\).
  Furthermore, let \(D\) the separable \(C^{\ast}\)-algebra generated by \(\{f\}\) and by the elements of \(\A\) viewed as constant functions. The elements of \(\A\) are limits of finite propagation operators in \(\chi_K\A\chi_K\) for \(K\) compact in \(\sigma(\M)\), thus \(\A\subset L^{\ast}\) and thus \(D\subset \mathcal{I}B^{\ast}\).

  Using lemma \ref{commutationlemma}, we can find \(x\in \mathcal{I}K^{\ast}\) such that \([x,d]\underset{t\to\infty}{\to} 0\) for all \(d\in D\) and \((1-x)d\underset{t\to\infty}{\to} 0\) for all \(d\in D\cap \mathcal{I}K^{\ast}\).

  Then setting \(\tilde{f}=f-xf\):
  \begin{itemize}
    \item \(f-\tilde{f}=xf\) is in \(\mathcal{I}C^{\ast}_{\spec,\BB}(\A)\). Indeed \(x\in \mathcal{I}K^{\ast}\), thus \(xf\) is a uniformly continuous path of operators in \(K^{\ast}\cdot C^{\ast}_{\spec,\BB}(\A)\subset C^{\ast}_{\spec,\BB}(\A)\).
    \item \(\tilde{f}\in D^{\ast}_{L,\spec,\BB}(\A)\). That it is in \(\mathcal{I}D^{\ast}_{\spec,\BB}(\A)\) is obvious from the above, furthermore, for \(a\in \A\),
          \[ [\tilde{f}(t),a]=[x,a]f(t)+(1-x)[f(t),a]\underset{t\to\infty}{\to} 0,\]
          by applying the conditions on \(x\) above to \(d=f\) and \(d=a\) respectively.
  \end{itemize}
\end{proof}

Using the following short term exact sequence:

\[
  0 \rightarrow C^{\ast}_{L,\spec,\BB}(\A) \rightarrow D^{\ast}_{L,\spec,\BB}(\A) \rightarrow \mathcal{I}D^{\ast}_{\spec,\BB}(\A)/\mathcal{I}C^{\ast}_{\spec,\BB}(\A)\rightarrow 0
\]
\vspace{0.1cm}

We obtain boundary maps \(K_{i+1}\left(\mathcal{I}D^{\ast}_{\spec,\BB}(\A)/\mathcal{I}C^{\ast}_{\spec,\BB}(\A)\right)\to K_i\left(C^{\ast}_{L,\spec,\BB}(\A)\right).\)
It suffices then to show that:
\begin{itemize}
  \item We have the vanishing formula: \(K_i\left(D^{\ast}_{L,\spec,\BB}(\A)\right)=0.\)
  \item The map induced by constant functions \(D^{\ast}_{\spec,\BB}(\A)/C^{\ast}_{\spec,\BB}(\A)\to\mathcal{I}D^{\ast}_{\spec,\BB}(\A)/\mathcal{I}C^{\ast}_{\spec,\BB}(\A)\) induces an isomorpism on K-theory: \[K_i\left(\mathcal{I}D^{\ast}_{\spec,\BB}(\A)/\mathcal{I}C^{\ast}_{\spec,\BB}(\A)\right)=K_i\left(D^{\ast}_{\spec,\BB}(\A)/C^{\ast}_{\spec,\BB}(\A)\right).\]
\end{itemize}

The proof of these two points relies on an Eilenberg swindle argument with the following two observations:
\begin{itemize}
  \item Independence of the Roe algebra of \(\BB\) ( recall \(\BB\) is commutative), and of the localization algebra on the representation used \(\rho\). (Take an intertwining path of isometries for the localization algebras and cut it by elements in \(\BB\) to get finite propagation with respect to \(\BB\)).
  \item The fact that if \(f\in D^{\ast}_{L,\spec,\BB}(\A)\) then \(\hat{f} \in \mathcal{I}\B(\HH^{\oplus \infty})\) defined by:
        \[\hat{f}(t):=\bigoplus_{n=0}^{\infty} f(t+n)\]
        is also in \(D^{\ast}_{L,\spec,\BB}(\A)\) with respect to the representation \(\bigoplus_{n=0}^{\infty} \rho\), since for fixed \(t\), \([\hat{f}(t),a]\) is a direct sum of compact operators of vanishing norm, it is compact.
\end{itemize}
The rest of the proof then follows from the same arguments as in the classical case (see \cite{qiao_roe}).
\section{Outlook}
In this paper we have established a link between locally compact quantum metric spaces and noncommutative coarse geometry, furthermore we have defined notions of Roe algebras and higher index theory in this framework. We believe that we have only scrached the surface of this new connection, and there are many open questions and directions that we leave for future research, in particular in connection with other developments in other fields of noncommutative geometry. We outline some of these directions below.
\subsection{More examples of spectral Roe algebra computations}
One direction is to compute the spectral Roe algebra for more examples of locally compact quantum metric spaces with more fundamental noncommutativity. One promising example is the Moyal plane, with the spectral triple strucuture defined on it in \cite{gayralMoyalPlanesAre2004} and considered as locally compact quantum metric space as in \cite{latremoliereQuantumLocallyCompact}. We expect that in such cases, the relative commutant Roe algebra and the spectral Roe algebra will differ significantly. Indeed the relative commutant Roe algebra, along with the noncommutative coarse structure arising from it, looks to be more natural in the noncommutative setting.
\subsection{Geometric realization of (finite) W*-quantum metric spaces} In relation with the point above, it would be interesting to study, even in the simplest cases (for instance \(\mathbb{M}_n(\mathbb{C})\) with a Lipschitz seminorm), what kind of W* quantum metric structure arise from the spectral propagation, and from the spetral propagation with respect to the diagonal subalgebra \(\mathbb{C}^n\), furthermore one could study what happens when different representations of \(\mathbb{M}_n(\mathbb{C})\) are considered. As arbitrary quantum graphs (and in particular those giving rise to quantum expanders) can be naturally viewed in the framework of quantum relations and thus of W* quantum metric spaces, this could allow us to construct ``geometric realizations'' of quantum graphs (and in particular quantum expanders) in the sense of Lipschitz seminorms.
\subsection{Spectral triples} In \cite{kuperbergNeumannAlgebraApproach2012}, W* quantum metric structures were associated to spectral triples, roughly by declaring that \(e^{itD}\) has propagation \(|t|\). In our construction, we have shown that locally compact quantum metric spaces with a faithful representation give rise to W* quantum metric spaces. It would be interesting to compare these two constructions, and see if they agree in some cases. Again we expect this to fail in general and that the relative commutant Roe algebra is the closer concept.
\subsection{Index theory}
While we have defined two types of higher indices for locally compact quantum metric spaces, and in parallel with the classical case, it would be interesting to study more concrete examples and compute these indices in some cases. The Moyal plane appears to be a promising example in this direction as well.
\subsection{\(\lip^{\ast}\) seminorms}
The seminorm \(\lip^{\ast}\) satisfies basic properties that make it a good candidate for a noncommutative Lipschitz seminorm, aside from the fact that \(\lip^{\ast}(T)=0\) does not imply that \(T\) is a scalar multiple of the identity. For finite dimensional algebras represented irreducibly on a finite dimensional Hilbert \(\HH_F\) space, this issue does not arise, and one gets a Lipschitz seminorm on \(\B(\HH_F)\). It would be interesting to study this seminorm in more detail, and see if it can be used to construct new examples of  (finite) quantum metric spaces, or if it can be used to give new insights on existing examples. For more general, nonfinite spaces, one can ask if one can take subspaces or quotient spaces of the relative commutant Roe algebra where \(\lip^{\ast}\) is a Lipschitz seminorm, and see if this can lead to new examples of quantum metric spaces.

\bibliographystyle{unsrt}
\bibliography{references}
\end{document}